\documentclass[smallextended,numbook]{svjour3} %numbook allow to use cleveref
\smartqed

\usepackage[english]{babel}
\usepackage{amsmath}
\usepackage{amssymb}
\usepackage{xfrac,enumitem}

\usepackage{tikz}
\usepackage{pgfplots}

\usepackage{graphicx}
\usepackage{mathtools}
\usepackage{booktabs}
\usepackage{xspace}
\usepackage[hidelinks]{hyperref}
\usepackage{subcaption}
\usepackage{todonotes}

\usepackage{thmtools}
\usepackage{thm-restate}

\usepackage{chngcntr}
\usepackage{cleveref}
\usepackage[percent]{overpic}

\usepackage{adjustbox}
\usepackage{comment}

\newenvironment{cpf}{\begin{trivlist} \item[] {\em Proof of Claim.}}{\hspace*{\stretch{1}} $\diamond$ \end{trivlist}}

\graphicspath{{pictures/}}

\definecolor{darkblue}{rgb}{0,0,.5}
\DeclareMathOperator{\cone}{cone} % cone
\DeclareMathOperator{\conv}{conv} % convex hull
\newcommand{\mbb}[1]{\mathbb{#1}}

\newcommand{\mcf}[1]{\mathcal{#1}}
\newcommand{\mbfs}[1]{\boldsymbol{#1}}

\DeclareMathOperator{\intr}{intr}

\newcommand{\T}{\top}
\newtheorem{Claim}{Claim}

\definecolor{MediumRed}{rgb}{0.925, 0.345, 0.345}
\definecolor{MediumGreen}{rgb}{0.37, 0.7, 0.66}
\definecolor{MediumBlue}{rgb}{0.015, 0.315, 0.45}
\definecolor{MediumPurple}{RGB}{153, 102, 203}
\definecolor{JungleGreen}{rgb}{0.16, 0.67, 0.53}

\usepackage{xifthen}
\ifthenelse{\isundefined{\T}}{%
  \newcommand{\T}{\mathsf{T}}
  }{%
  \renewcommand{\T}{\mathsf{T}}
}

\newtheorem{rmk}[remark]{Remark}
\title{A characterization of maximal homogeneous-quadratic-free sets\footnote{An extended abstract of this article appeared at IPCO 2023~\cite{MPS2023}. This full version contains many new results, including a resolution of a conjecture posed in the extended abstract.}}
\author{
  Gonzalo Mu\~noz \and
  Joseph Paat\and
  Felipe Serrano
}
\institute{
G. Mu\~noz \at
Engineering Sciences Institute, Universidad de O'Higgins, Rancagua, Chile \\
  \email{gonzalo.munoz@uoh.cl} 
  \and
J. Paat \at
Sauder School of Business, University of British Columbia, Vancouver
BC, Canada \\
  \email{joseph.paat@sauder.ubc.ca} 
  \and
  F. Serrano \at
  COPT GmbH, Berlin, Germany \\
  \email{serrano@copt.de}
}

\date{}

\begin{document}
%\zibtitlepage
\maketitle

\maketitle
\begin{abstract}

The intersection cut framework was introduced by Balas in 1971 as a method for generating cutting planes in integer optimization.
In this framework, one uses a full-dimensional convex $S$-free set, where $S$ is the feasible region of the integer program, to derive a cut separating $S$ from a non-integral vertex of a linear relaxation of $S$.
Among all $S$-free sets, it is the inclusion-wise maximal ones that yield the strongest cuts.
Recently, this framework has been extended beyond the integer case in order to obtain cutting planes in non-linear settings.
In this work, we consider the specific setting when $S$ is defined by a homogeneous quadratic inequality.
In this `quadratic-free' setting, every function $\Gamma: D^m \to D^n$, where $D^k$ is the unit disk in $\mbb{R}^k$, generates a representation of a quadratic-free set. 
While not every $\Gamma$ generates a maximal quadratic free set, it is the case that every full-dimensional maximal quadratic free set is generated by some $\Gamma$.
Our main result shows that the corresponding quadratic-free set is full-dimensional and maximal if and only if $\Gamma$ is non-expansive and satisfies a technical condition.
This result yields a broader class of maximal $S$-free sets than previously known.
Our result stems from a new characterization of maximal $S$-free sets (for general $S$ beyond the quadratic setting) based on sequences that `expose' inequalities defining the $S$-free set.

\smallskip
\noindent \textbf{Keywords.} Quadratic programming, cutting planes, intersection cuts, $S$-free sets

\end{abstract}%

\section{Introduction}

Given a closed set $S \subseteq \mbb{R}^d$, we say that a closed convex set $C \subseteq \mbb{R}^d$ is {\it $S$-free} if its interior $\intr(C)$ contains no points from $S$.
The family of $S$-free sets forms the foundation of {\it intersection cuts} for mathematical programs of the form
\begin{equation}\label{eq:mainprob}
\min \{\mbfs{c}{}^\T \mbfs{s}:\ \mbfs{s} \in S\},
\end{equation}
where $\mbfs{c} \in \mbb{R}^d$.
The overall intersection cuts framework operates as follows: if one solves a linear relaxation of \eqref{eq:mainprob} and obtains a vertex $\mbfs{s}\not\in S$, then the construction of an $S$-free set containing $\mbfs{s}$ in its interior ensures the separation of $\mbfs{s}$.
For a general reference on intersection cuts including more on the precise derivation of an intersection cut from an $S$-free set, we point to~\cite[Chapter 6]{CCZ2014}.
Intersection cuts were introduced by Balas~\cite{B1971} when $S$ is a lattice and by Tuy~\cite{T1964} when $S$ is
a reverse convex set.
Since then, intersection cuts have been well-studied; see, e.g.,~\cite{AJ2013,ALWW2007,ABP2018,BCM2019,BCM2020,CCDLM2014,MKV2016}.

Inclusion-wise {\it maximal} $S$-free sets play an important role as they generate the strongest intersection cuts.
Through the lens of mixed-integer duality, maximal $S$-free sets can also serve as optimality certificates for mixed-integer programs; see~\cite{BOW2016,BCCWW2017,PSS2022}.
For the case when $S$ is a lattice, Lov\'{a}sz~\cite{L1989} demonstrates that full-dimensional maximal $S$-free sets are polyhedra with integer points in the relative interior of each facet; see also~\cite{A2013}. 
See~\cite{BCCZ2010,basu2010minimal} for extensions of Lov\'{a}sz's result to lattice points in linear subspaces and rational polyhedra.
Connections between maximal $S$-free sets and the Helly number of $S$ have been established in~\cite{averkov2013maximal,conforti2016maximal,BCCWW2017}.
There are some characterizations of maximality beyond the lattice setting, e.g.,~\cite{BCM2019,BCM2020} characterize various maximal $S$-free sets when $S$ is the set of rank 1 real-valued symmetric matrices. 

Mu\~{n}oz and Serrano~\cite{MS2020,MS2022} are the first to study the setting when $S$ is defined by an arbitrary quadratic inequality. 
They develop methods for proving maximality when $S$ is defined by either a single homogeneous or a single non-homogeneous quadratic inequality. 
A computational implementation of the resulting intersection cuts is developed in \cite{chmiela2022implementation}, with favorable results.
In this paper, we focus on the homogeneous quadratic setting
\[
S = \left\{ \mbfs{s} \in \mbb{R}^{d}  : \mbfs{s}{}^\T \mbfs{A} \mbfs{s} \le0\right\},
\]
where $\mbfs{A} \in \mbb{R}^{d\times d}$.
We note that maximal $S$-free sets derived in this setting extend to (not necessarily maximal) sets that can be used in the non-homogeneous setting as well. 
Indeed, an $\widehat{S}$-free set for the non-homogeneous setting
\[
\widehat{S} := \left\{ \mbfs{s} \in \mbb{R}^{d}  : \mbfs{s}{}^\T \mbfs{A} \mbfs{s} +  \mbfs{g}^\T \mbfs{s} +  h \le0\right\}
\]
can be constructed by taking an $S'$-free set for the homogeneous set 
\[
S':=  \left\{ (\mbfs{s},z) \in \mbb{R}^{d}\times \mbb{R} : \mbfs{s}{}^\T \mbfs{A} \mbfs{s} +  (\mbfs{g}^\T \mbfs{s})z +  hz^2 \le0\right\}
\]
and intersecting the $S'$-free set with $z=1$.
To simplify our presentation in the homogeneous setting, we follow reductions in~\cite{MS2020,MS2022} to assume, without significant loss of generality, that the homogeneous setting has the form
\[
Q := \left\{ (\mbfs{x},\mbfs{y}) \in \mbb{R}^n \times \mbb{R}^m : \|\mbfs{x}\| \le \|\mbfs{y}\|\right\},
\]
where $\|\cdot\|$ is the $\ell_2$-norm. 
We replace $S$ (resp. $S$-free) with $Q$ (resp. $Q$-free) to highlight that we are looking at {\it quadratic-free sets}.
We use this definition of $Q$ for the rest of the paper,
and refer to any set $C$ that is $Q$-free as {\it homogeneous quadratic-free}.

Among their results, Mu\~{n}oz and Serrano prove that a particular homogeneous quadratic-free set is maximal~\cite[Theorem 2.1]{MS2022}.
One of the motivations in this paper is to provide more general characterizations of maximal $Q$-free sets that can be used to generate alternative families of them and, consequently, new families of cutting planes for quadratically-constrained problems. 

In this paper, we demonstrate that every maximal $Q$-free set has a special form $C_{\Gamma}$ for a particular function $\Gamma$; see~\eqref{Cgammadef}.
Then, we characterize maximality of $C_{\Gamma}$ based on properties of $\Gamma$.
In order to derive our characterizations, we also derive a new characterization of maximality for $S$-free sets when $S$ is an arbitrary closed set.

\medskip

\noindent{\bf Notation.} 
Suppose $\mcf{X} \subseteq \mbb{R}^d$.
We use $\conv(\mcf{X})$ to denote the convex hull of $\mcf{X}$.
We use $\cone(\mcf{X}) := \{\sum_{i=1}^t \lambda_i \mbfs{x}^i:\ t\in \mbb{N},\ \lambda_i \ge 0~\text{and}~\mbfs{x}^i \in \mcf{X} ~\forall ~i \in \{1, \dotsc, t\}\}$ to denote the conic hull of $\mcf{X}$.
Similarly, we use $\overline{\conv}(\mcf{X})$ and $\overline{\cone}(\mcf{X})$ to denote the closed convex hull and closed conic hull of $\mcf{X} $, respectively.
For background on convexity including common definitions, we point to~\cite{R1970}.
We use bold font to denote vectors.
We set $D^d := \{\mbfs{x} \in \mbb{R}^d:\ \|\mbfs{x}\| = 1\}$.

\subsection{Contributions} \label{sec:contributions}

In order to study full-dimensional $Q$-free sets, we consider a special family of sets parameterized by functions of the form $\Gamma : D^m \to D^n$.
Given such a $\Gamma$, define the set
\begin{equation} \label{Cgammadef}
C_\Gamma := \{(\mbfs{x},\mbfs{y}) \in \mbb{R}^n \times \mbb{R}^m : \Gamma(\mbfs{\beta})^\T \mbfs{x} - \mbfs{\beta} ^\T \mbfs{y} \ge 0 ~~ \forall ~ \mbfs{\beta} \in D^m\}.
\end{equation}
Sets of this form play an important role for us because for every $\Gamma$, the set $C_{\Gamma}$ is $Q$-free.
Indeed, by the Cauchy-Schwarz inequality any $(\mbfs{x},\mbfs{y})$ in the interior of $C_\Gamma$ satisfies $\mbfs{\beta} ^\T \mbfs{y} < \|  \mbfs{x}\| $ for all $\mbfs{\beta} \in D^m$ which implies that $\|\mbfs{y}\| < \|\mbfs{x}\|$.
We have observed the following:

\begin{rmk}[Every $C_{\Gamma}$ is $Q$-free]\label{rmkQFree}
Let $\Gamma : D^m \to D^n$ and $C_{\Gamma}$ be defined as in~\eqref{Cgammadef}.
The set $C_\Gamma$ is closed, convex and $Q$-free.
\end{rmk}

Moreover, not only is $C_{\Gamma}$ $Q$-free for every $\Gamma$, but interestingly every full-dimensional maximal $Q$-free set can be written in the `standard form'~\eqref{Cgammadef}.
This constitutes our first main result. 

\begin{theorem}[Standard form is necessary for maximality]\label{thmStdForm}
Let $C$ be a full-dimensional closed convex maximal $Q$-free set. 
There exists a function $\Gamma : D^m \to D^n$ such that $C=C_\Gamma$.
\end{theorem}

\begin{rmk}{\bf(Lower-dimensional maximal $Q$-free sets are always hyperplanes)}\label{rmkLowDim}
As with many of our results, Theorem~\ref{thmStdForm} focuses on sets that are {\it full-dimensional}. 
One reason why we emphasize full-dimensionality is because lower-dimensional maximal $Q$-free sets are necessarily hyperplanes. 
Indeed, a lower-dimensional maximal $Q$-free sets is contained in a hyperplane $H$, which has empty interior and is therefore trivially $Q$-free.
\end{rmk}

In light of Theorem~\ref{thmStdForm}, the question of characterizing full-dimensional maximal $Q$-free sets can be reduced to characterizing the functions $\Gamma$ corresponding to maximal $C_{\Gamma}$.
We demonstrate that a meaningful property in this direction is the `expansivity' of $\Gamma$.

\begin{definition}{\bf(Expansive, Non-expansive, isometric, strictly non-expansive)}
Let $\Gamma: D^m \to D^n$.
Two points $\mbfs{\beta}, \mbfs{\beta}' \in D^m$ are
\begin{itemize}[leftmargin = *]
\item {\bf expansive} if $\|\Gamma(\mbfs{\beta}) - \Gamma(\mbfs{\beta}')\| > \|\mbfs{\beta} - \mbfs{\beta}'\|$.
\item {\bf isometric} if $\|\Gamma(\mbfs{\beta}) - \Gamma(\mbfs{\beta}')\| = \|\mbfs{\beta} - \mbfs{\beta}'\|$.
\item {\bf non-expansive} if $\|\Gamma(\mbfs{\beta}) - \Gamma(\mbfs{\beta}')\| \le \|\mbfs{\beta} - \mbfs{\beta}'\|$.
\item {\bf strictly non-expansive} $\|\Gamma(\mbfs{\beta}) - \Gamma(\mbfs{\beta}')\| < \|\mbfs{\beta} - \mbfs{\beta}'\|$.
\end{itemize}
The function $\Gamma$ is {\bf non-expansive} (respectively, {\bf strictly non-expansive}) if every pair of distinct points $\mbfs{\beta}, \mbfs{\beta}' \in D^m$ are non-expansive (respectively, strictly non-expansive).
\end{definition}

Observe that for a function $\Gamma: D^m \to D^n$, two points $\mbfs{\beta}, \mbfs{\beta}' \in D^m$ are non-expansive (respectively, expansive, isometric, or strictly non-expansive) if and only if $\mbfs{\beta}{}^\T\mbfs{\beta}' \le \Gamma(\mbfs{\beta})^\T\Gamma(\mbfs{\beta}')$ (respectively, $\mbfs{\beta}{}^\T\mbfs{\beta}' > \Gamma(\mbfs{\beta})^\T\Gamma(\mbfs{\beta}')$, $\mbfs{\beta}{}^\T\mbfs{\beta}' = \Gamma(\mbfs{\beta})^\T\Gamma(\mbfs{\beta}')$,  or $\mbfs{\beta}{}^\T\mbfs{\beta}' < \Gamma(\mbfs{\beta})^\T\Gamma(\mbfs{\beta}')$) because $\mbfs{\beta}, \mbfs{\beta}' \in D^m$ and $\Gamma(\mbfs{\beta}), \Gamma(\mbfs{\beta}') \in D^n$.

\begin{rmk}\label{rmk:contracting}
The reader may be familiar with the notion of a {\bf contracting} function. 
The function $\Gamma$ is contracting if there exists $L<1$ such that
\[\|\Gamma(\mbfs{\beta}^1) - \Gamma(\mbfs{\beta}^2)\| \leq L \|\mbfs{\beta}^1 - \mbfs{\beta}^2\| \quad \forall \ \mbfs{\beta}^1, \mbfs{\beta}^2 \in D^m.\]
If $\Gamma$ is contracting, then it is strictly non-expansive. 
However, the converse is not true. 
Consider $n=m=2$ and define $\Gamma$ using polar coordinates: for $\theta\in[0,2\pi]$, we define $\gamma(\theta) = - \theta(\theta - 2\pi)/(4\pi)$ and define $\Gamma$ such that 
\[
\beta(\theta) := (\cos(\theta),\sin(\theta)) \mapsto \Gamma(\beta(\theta)) := (\cos(\gamma(\theta)),\sin(\gamma(\theta))).
\]
$\Gamma$ is strictly non-expansive; we omit a full technical proof of this, but we illustrate this fact in Figure \ref{fig:nonpolyex-theta}.
Let us now argue that $\Gamma$ is not contracting.
Take $\mbfs{\beta}^1= (1,0)$ and $\mbfs{\beta}^2(\theta) = (\cos(\theta), \sin(\theta))$. 
Note that $\Gamma(\mbfs{\beta}^1) = (1,0)$ and thus
\[
\lim_{\theta\to 0} \frac{\Gamma(\mbfs{\beta}^1)^\T \Gamma(\mbfs{\beta}^2)}{\mbfs{\beta}^1{}^\T \mbfs{\beta}^2} = \lim_{\theta\to 0} \frac{\cos(- \theta(\theta - 2\pi)/(4\pi))}{\cos(\theta)} =  1.
\]
This implies $\lim_{\theta\to 0} \|\Gamma(\mbfs{\beta}^1) - \Gamma(\mbfs{\beta}^2) \|/\|\mbfs{\beta}^1 -  \mbfs{\beta}^2\| = 1 $, and thus $\Gamma$ cannot be contracting.
\end{rmk}

\begin{figure}[t]
		\centering
		\includegraphics[scale=0.30]{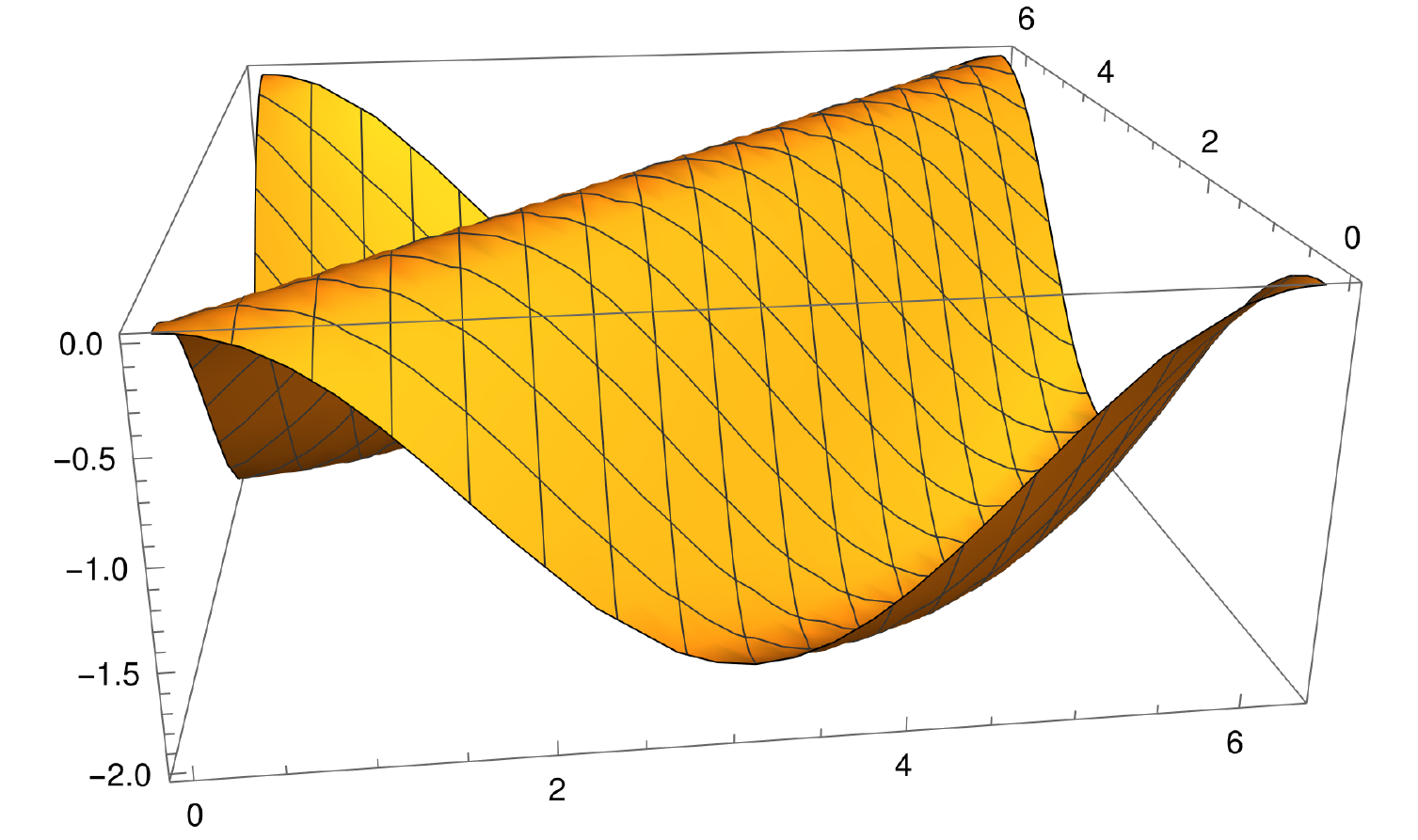}
		\caption{Plot of $\beta(\theta_1)^\T\beta(\theta_2) - \Gamma(\beta(\theta_1))^\T\Gamma(\beta(\theta_2))$ for $\Gamma$ defined in Remark \ref{rmk:contracting}. This quantity is non-positive and takes the value 0 only when $\theta_1 = \theta_2$ (mod $2\pi$), which shows $\Gamma$ is strictly non-expansive. }
		\label{fig:nonpolyex-theta}
\end{figure}

Our second main result characterizes functions $\Gamma$ that yield full-dimensional $Q$-free sets.

\begin{theorem}[Characterizing maximality of $C_\Gamma$]\label{thm:continuous}
Let $\Gamma: D^m \to D^n$ and define $C_\Gamma$ as in \eqref{Cgammadef}. 
The set $C_{\Gamma}$ is a full-dimensional maximal $Q$-free set if and only if $\Gamma$ is non-expansive and $(\mbfs{0}, \mbfs{0})\not\in \conv(\{(\Gamma(\mbfs{\beta}),-\mbfs{\beta})\,:\ \mbfs{\beta}\in D^m\})$.
\end{theorem}

The backwards direction in Theorem~\ref{thm:continuous} generalizes the work in~\cite[Theorem 3]{MPS2023}, where the additional assumption that $C_{\Gamma}$ is polyhedral is needed.
Furthermore, the result settles the conjecture stated prior to Theorem 3 in their work. 
The forward direction in Theorem~\ref{thm:continuous} is new to this work, and it allows for a complete characterization of maximality.
Note that Theorem~\ref{thm:continuous} implies that a discontinuous $\Gamma$ cannot yield a maximal $Q$-free set.

For the case when $\Gamma$ is strictly non-expansive, we can prove that $C_{\Gamma}$ is a full-dimensional maximal $Q$-free set without
any extra condition.
Indeed, if $\Gamma$ is strictly non-expansive, then for each $\mbfs{\beta}'$, the point $(\Gamma(\mbfs{\beta}'), \mbfs{\beta}')$ strictly satisfies every defining inequality for $C_{\Gamma}$ except $\Gamma(\mbfs{\beta}'){}^\top \mbfs{x} - \mbfs{\beta}'{}^\T \mbfs{y} \ge 0$; the set $\{(\Gamma(\mbfs{\beta}), \mbfs{\beta}):\ \mbfs{\beta} \in D^m\}$ is compact because $\Gamma$ is continuous; thus, there is some $\delta > 0$ such that $\Gamma(\mbfs{\beta}){}^\top \Gamma(\mbfs{\beta}') - \mbfs{\beta}{}^\top \mbfs{\beta}' \ge \delta $; hence, the midpoint of distinct points $\mbfs{\beta}', \mbfs{\beta}'' \in D^m$ satisfies all inequalities strictly and with some slack; thus, $C_{\Gamma}$ is full-dimensional. 
In Lemma~\ref{lemNoZeroSameCoeff} we show that if $C_{\Gamma}$ is full-dimensional and $\Gamma$ is continuous, then $(\mbfs{0}, \mbfs{0})\not\in \conv(\{(\Gamma(\mbfs{\beta}),-\mbfs{\beta})\,:\ \mbfs{\beta}\in D^m\})$.
Together with Theorem~\ref{thm:continuous}, we have arrived at the following corollary.

\begin{corollary}[Strictly non-expansive implies full-dimensionality]\label{corMain2}
Let $\Gamma: D^m \to D^n$ and define $C_\Gamma$ as in \eqref{Cgammadef}.
If $\Gamma$ is strictly non-expansive, then $C_\Gamma$ is a full-dimensional maximal $Q$-free set.
\end{corollary}

\noindent Corollary~\ref{corMain2} generalizes~\cite[Theorem 2.1]{MS2022}, where $\Gamma$ is a constant function.
Example \ref{ex1} in Section \ref{secExamples} illustrates the construction of a maximal $Q$-free set using a non-constant $\Gamma$ function that is strictly non-expansive.

Another result that we find to be of independent interest when studying $Q$-free polyhedra is the following characterization of polyhedrality depending on $\Gamma$.

\begin{theorem}[A characterization of polyhedrality]\label{thmCpolyhedral}
Let $\Gamma: D^m \to D^n$ be non-expansive and define $C_\Gamma$ as in \eqref{Cgammadef}.
$C_\Gamma$ is a polyhedron if and only if there is a finite set $I \subseteq D^m$ such that for every $\overline{\mbfs{\beta}} \in D^m$ there exists a set $J\subseteq I$ of pairwise isometric points satisfying $\overline{\mbfs{\beta}} \in \cone(J)$.
Moreover, $C_\Gamma =  \{(\mbfs{x},\mbfs{y}) \in \mbb{R}^n \times \mbb{R}^m : \Gamma(\mbfs{\beta})^\T \mbfs{x} - \mbfs{\beta} ^\T \mbfs{y} \ge 0 ~~ \forall ~ \mbfs{\beta} \in I\}$.
\end{theorem}

Examples \ref{ex2} and \ref{ex3} in Section~\ref{secExamples} show the construction of polyhedral maximal $Q$-free sets using a non-expansive $\Gamma$.
We point out one difference between maximal $Q$-free polyhedra and maximal $S$-free polyhedra when $S$ is a lattice.
In Lov\'{a}sz's characterization, the number of facets of a maximal $S$-free polyhedra is upper bounded by a function of the dimension; this is not the case for maximal $Q$-free polyhedra, which can have an arbitrary number of facets 
(see Example \ref{ex3} in Section~\ref{secExamples}).
We also note that Theorem \ref{thmCpolyhedral} can be used to construct, starting from a set $I\subseteq D^m$, 
a function $\Gamma$ that yields a maximal $Q$-free polyhedron (see Example \ref{ex3} for an illustration of this).

Underlying the proof of Theorem~\ref{thm:continuous} is our final main result, which is a general characterization of maximality of $S$-free sets.
We turn once again to the case when $S$ is a lattice to motivate this result.
Lov\'{a}sz proves that if $C$ is maximal $S$-free, then $C$ is a polyhedron and every facet $F$ of $C$ contains a lattice point $\mbfs{z}^F$ in its relative interior. 
The point $\mbfs{z}^F$ is similar to a `blocking point' used to generate maximal $S$-free sets when $S$ is a mixed-integer set through lifting~\cite{BDP2019,CCZ2011,DW2010}.
In order for $C$ to be $S$-free in the lattice setting, each $\mbfs{z}^F$ must be separated from $C$ by a facet defining inequality. 
In fact, $F$ is the unique facet separating $\mbfs{z}^F$ from $C$; thus, in a way $\mbfs{z}^F$ `exposes' $F$. 
The notion of exposing points is considered by Mu\~{n}oz and Serrano~\cite{MS2022}, where they argue that if every inequality defining a $Q$-free set $C$ has an exposing point, then $C$ is maximal. 
However, there are maximal $Q$-free sets defined by inequalities that do not have exposing points; see Example \ref{ex2} in Section \ref{secExamples}.
A generalization of this is the notion of an exposing sequence.

\begin{definition}[Exposing sequence]\label{def:exposingseq}
  Let $C \subseteq \mbb{R}^d$ be a convex set and let $\mbfs{\alpha}^\T \mbfs{x} \le \alpha_0$, with $\mbfs{\alpha} \neq  \mbfs{0}$, be a valid inequality for $C$.
  A sequence $(\mbfs{x}^t)_{t=1}^\infty$ in $\mbb{R}^d$ is an {\bf exposing sequence} for $\mbfs{\alpha}^\T \mbfs{x} \le \alpha_0$ if $\lim_{t\to\infty}(\mbfs{\delta}^t, {\delta_0^t}) = (\mbfs{\alpha}, \alpha_0)$ for every sequence $((\mbfs{\delta}^t, {\delta^t_0}))_{t=1}^\infty$ in ${\mbb{R}}^d \times \mbb{R}$ such that $\|\mbfs{\delta}^t\| = \|\mbfs{\alpha}\|$, $\mbfs{\delta}^t{}^\T \mbfs{x} \le \delta^t_0$ is a valid inequality for $C$, and ${\mbfs{\delta}^t}{}^\T \mbfs{x}^t \ge {\delta^t_0}$ for each $t$.
\end{definition}

\begin{rmk}
Mu\~{n}oz and Serrano~\cite{MS2022} define a notion of `exposing sequence at infinity'. 
This is more restrictive than Definition \ref{def:exposingseq}.
Moreover, it can be shown that an exposing sequence at infinity reduces to an exposing point for $Q$ because $Q$ is a cone. 
\end{rmk}

\begin{theorem} \label{thmMax}
Let $S \subseteq \mbb{R}^d$ be closed, and let $C \subseteq \mbb{R}^d$ be a closed convex full-dimensional $S$-free set.
$C$ is a maximal $S$-free set if and only if there exists a set $I \subseteq \mbb{R}^d\times \mbb{R}$ such that
  \[
    C = \{ \mbfs{x} \in \mbb{R}^d : \mbfs{\alpha}^\T \mbfs{x} \leq \alpha_0 ~~ \forall\  (\mbfs{\alpha}, \alpha_0) \in I \}
  \]
and each $(\mbfs{\alpha}, \alpha_0) \in I$ has an exposing sequence $(\mbfs{x}^t)_{t=1}^\infty$ in $S$.
\end{theorem}

The rest of the paper is outlined as follows. 
We give examples of maximal $Q$-free sets in Section~\ref{secExamples}.
We prove the general characterization Theorem~\ref{thmMax} in Section~\ref{secExposing}.
The necessity of the standard form $C_{\Gamma}$ (Theorem~\ref{thmStdForm}) is given in Section~\ref{secHomogeneous}.
Preliminary results from convex analysis and about non-expansive functions are given in Section~\ref{secNEF}.
We prove the characterization of maximality (Theorem~\ref{thm:continuous}) in Section~\ref{secMainThm}.
We prove Theorem~\ref{thmMax} in Section~\ref{secPolyhedral}.
Finally, we discuss future work in Section~\ref{secFutureWork}.

%%%%%%%%%%%%%%%%%%%

\section{Examples of maximal $Q$-free sets}\label{secExamples}

\begin{example}\label{ex1}
Our first example illustrates a non-polyhedral $Q$-free set using Corollary \ref{corMain2}.
Consider $n=m=2$. We construct the function $\Gamma$ described in Remark \ref{rmk:contracting}: using polar coordinates, for $\theta\in[0,2\pi]$, we define $\gamma(\theta) = - \theta(\theta - 2\pi)/(4\pi)$ and define $\Gamma$ such that $\beta(\theta) := (\cos(\theta),\sin(\theta)) \mapsto \Gamma(\beta(\theta)) := (\cos(\gamma(\theta)),\sin(\gamma(\theta)))$.
As discussed in Remark \ref{rmk:contracting}, $\Gamma$ is strictly non-expansive. 
\begin{figure}[t]
		\centering
		\includegraphics[scale=0.25]{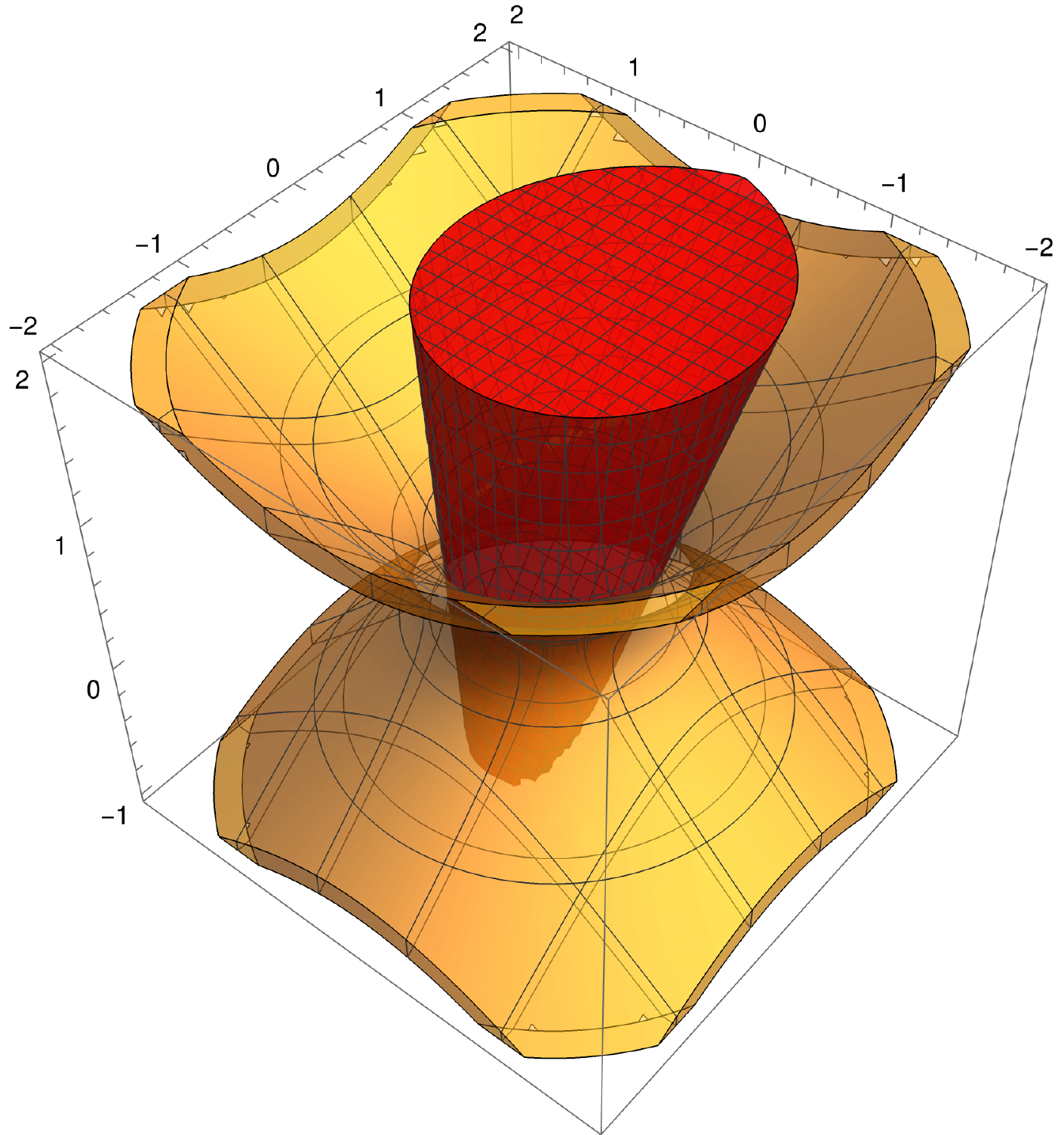}
		\caption{$3$-dimensional slices of the 4-dimensional sets $Q$ (boundary in orange) and $C_\Gamma$ (red) of Example \ref{ex1}.\\[.35 cm]}
		\label{fig:nonpolyex}
	%\label{fig:ex3}
\end{figure}
Therefore, Corollary \ref{corMain2} implies that $C_\Gamma$ is a full-dimensional maximal $Q$-free set. Figure \ref{fig:nonpolyex} shows a 3-dimensional slice of this 4-dimensional $C_\Gamma$. Note how the non-differentiability of $\gamma$ at 0 translates into a non-smooth region of the set.
\end{example}

Our next example shows an application of Theorem~\ref{thm:continuous} in a non-polyhedral setting.

\begin{example}\label{ex4}
Consider $n=m=2$. 
Similarly to Remark \ref{rmk:contracting}, we construct a function $\Gamma$ using polar coordinates. 
For $\theta\in[0,2\pi]$, define
 
\[\gamma(\theta) =
 \left\{ 
\begin{array}{cc}
 \theta & 0\leq \theta\leq \frac{\pi }{2} \\
 \pi -\theta & \frac{\pi }{2}\leq \theta\leq \pi  \\
 \theta-\pi  & \pi \leq \theta\leq \frac{5 \pi }{4} \\
 \frac{1}{3} (2 \pi -\theta) & \frac{5 \pi }{4}\leq \theta\leq 2 \pi  \\
\end{array}\right.
\]
 and define $\Gamma$ such that 
\[\mbfs{\beta}(\theta) := (\cos(\theta),\sin(\theta)) \mapsto \Gamma(\mbfs{\beta}(\theta)) := (\cos(\gamma(\theta)),\sin(\gamma(\theta))).\] 
 In Figure \ref{fig:nonpolyex-new} we show plots of $\gamma$ and of $\mbfs{\beta}(\theta_1)^\T\mbfs{\beta}(\theta_2) - \Gamma(\mbfs{\beta}(\theta_1))^\T\Gamma(\mbfs{\beta}(\theta_2))$ that illustrate that $\Gamma$ is non-expansive. 
 Note that $\Gamma$ is not strictly non-expansive: this can be seen in Figure \ref{fig:subfiex4a}, as $\mbfs{\beta}(\theta_1)^\T\mbfs{\beta}(\theta_2) - \Gamma(\mbfs{\beta}(\theta_1))^\T\Gamma(\mbfs{\beta}(\theta_2)) = 0$ for some values $\theta_1 \neq \theta_2$ (mod $2\pi$).
Additionally, $(\mbfs{0}, \mbfs{0})\not\in \conv(\{(\Gamma(\mbfs{\beta}),-\mbfs{\beta})\,:\ \mbfs{\beta}\in D^2\})$, as $\gamma(\theta) \in [0,\pi/2]$ and thus $\Gamma(\mbfs{\beta})$ lies on an arc segment in the non-negative orthant.
\begin{figure}[t]
\centering
     \begin{subfigure}[t]{0.45\linewidth}
\centering
\includegraphics[scale=0.25]{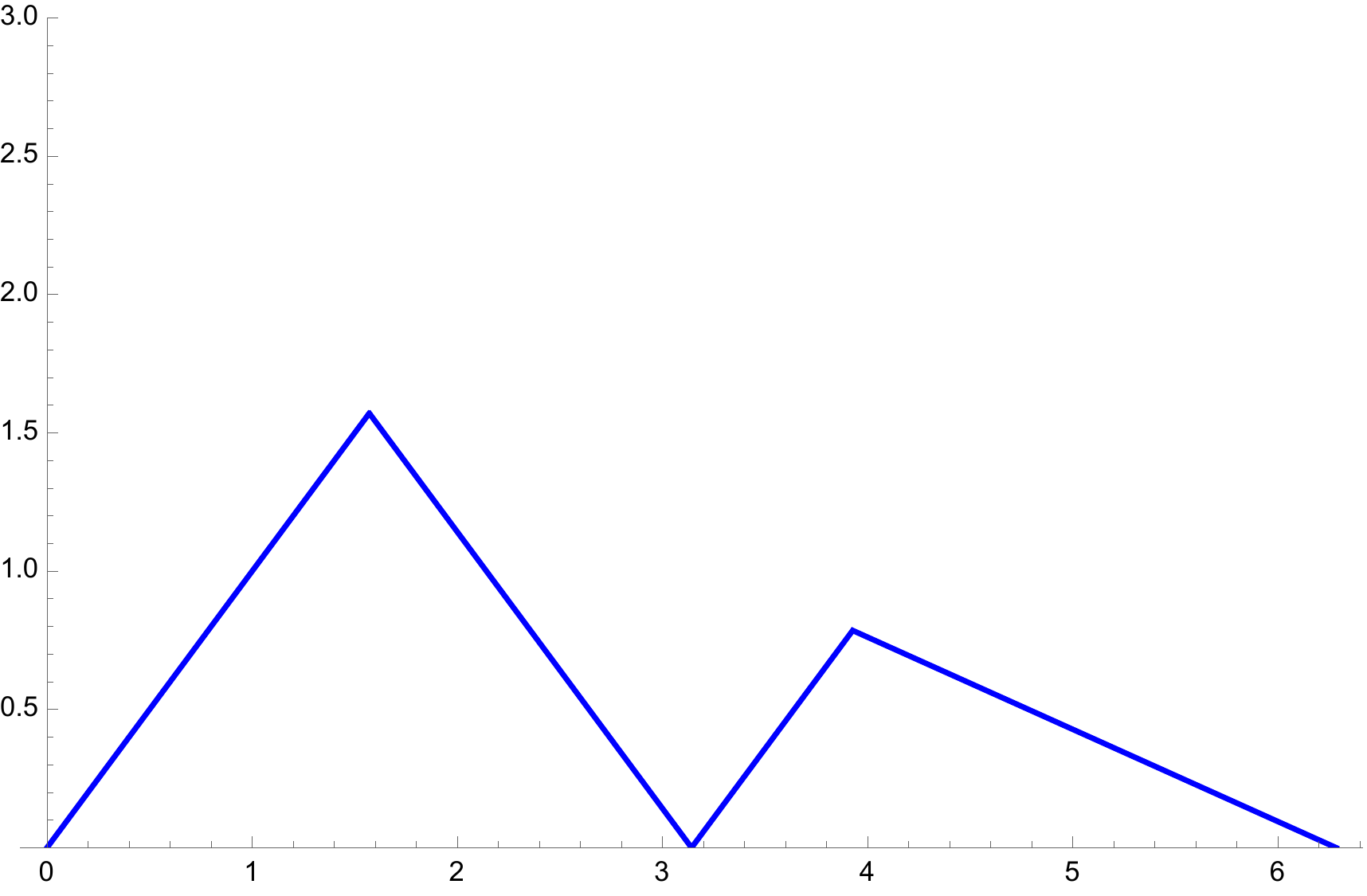}
    \caption{Plot of $\gamma$ function in Example \ref{ex4}. This function is a representation of the $\Gamma$ function in polar coordinates.}
    \label{fig:subfiex4a}
\end{subfigure}%
\hfill
     \begin{subfigure}[t]{0.45\linewidth}
\centering
\includegraphics[scale=0.25]{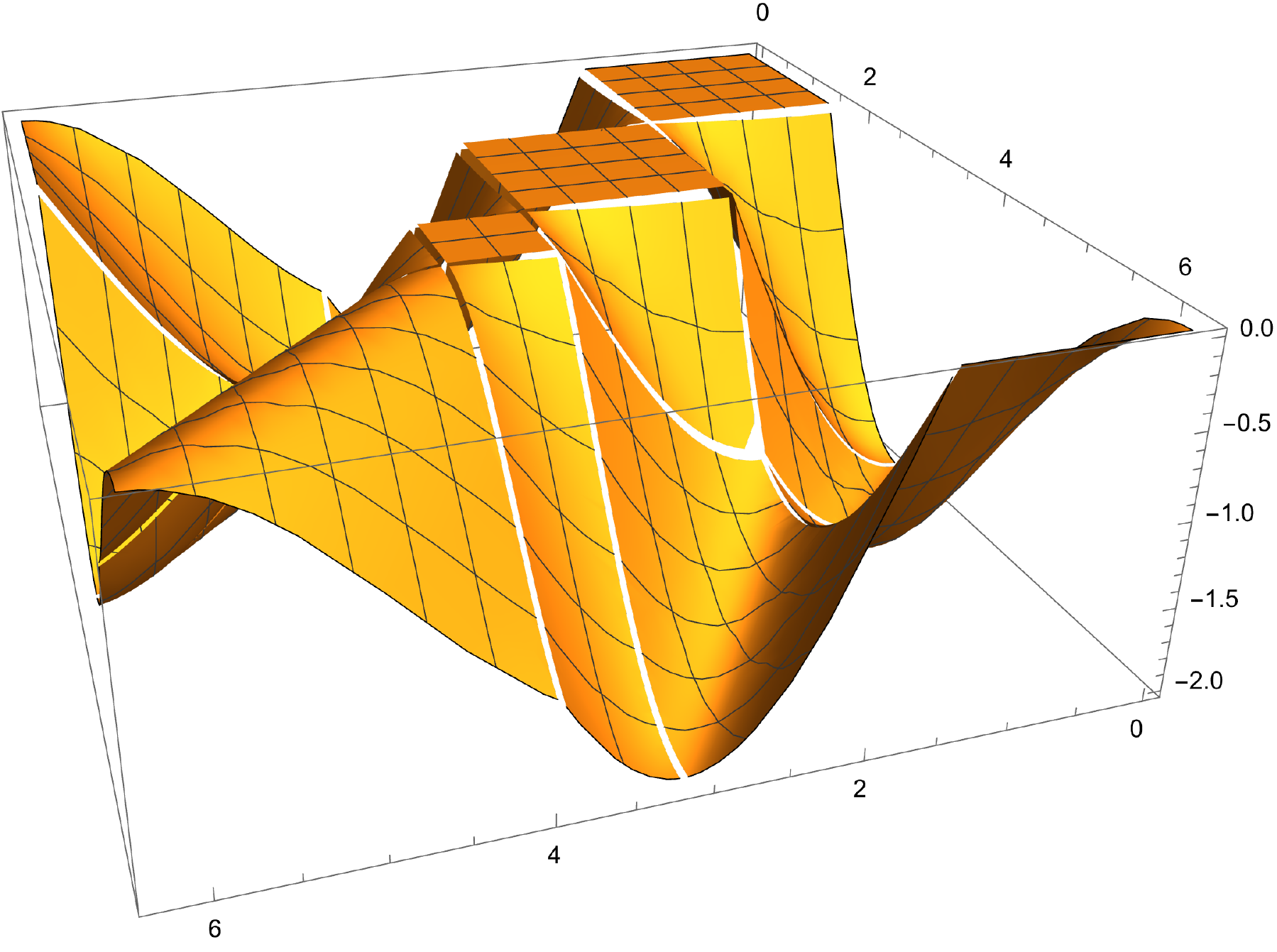}
    \caption{Plot of $\mbfs{\beta}(\theta_1)^\T\mbfs{\beta}(\theta_2) - \Gamma(\mbfs{\beta}(\theta_1))^\T\Gamma(\mbfs{\beta}(\theta_2))$ for $\Gamma$ defined in Example \ref{ex4}. This quantity is non-positive.}
    \label{fig:subfiex4a}
\end{subfigure}
		\caption{Plots of $\gamma$ and $\mbfs{\beta}(\theta_1)^\T\mbfs{\beta}(\theta_2) - \Gamma(\mbfs{\beta}(\theta_1))^\T\Gamma(\mbfs{\beta}(\theta_2))$ in Example \ref{ex4} which lead to a maximal quadratic-free set.}
		\label{fig:nonpolyex-new}
\end{figure}
Theorem~\ref{thm:continuous} implies that $C_\Gamma$ is a full-dimensional maximal $Q$-free set. Figure \ref{fig:nonpolyex-new-slice} shows a 3-dimensional slice of this 4-dimensional $C_\Gamma$. 
\begin{figure}[t]
		\centering
		\includegraphics[scale=0.20]{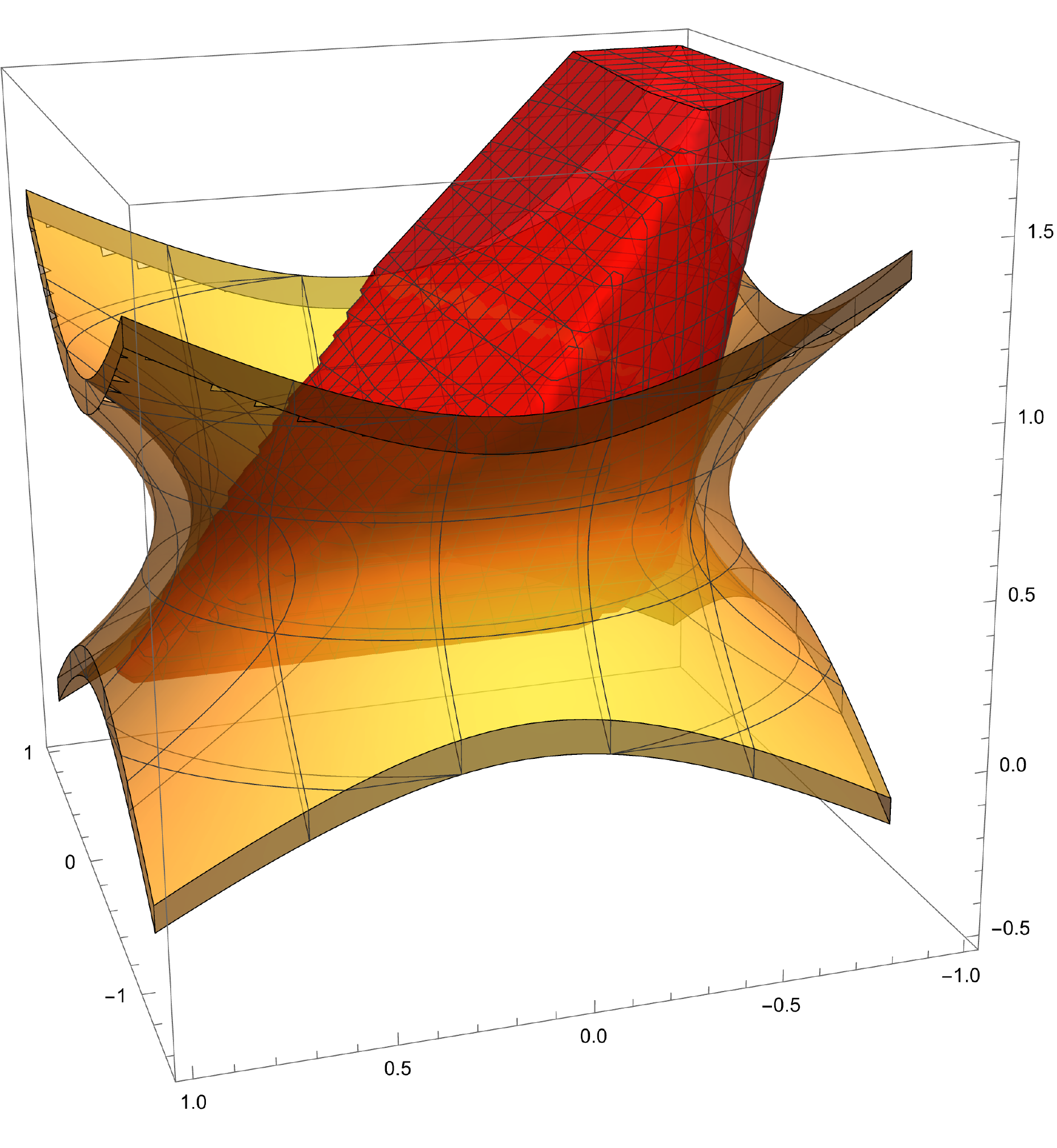}
		\caption{$3$-dimensional slices of the 4-dimensional sets $Q$ (boundary in orange) and $C_\Gamma$ (red) of Example \ref{ex4}. Maximality may not be evident in this picture because maximality is not necessarily preserved when taking slices.}
		\label{fig:nonpolyex-new-slice}
	%\label{fig:ex3}
\end{figure}
\end{example}

\begin{example}\label{ex2}
Suppose $n=m$ and define $\Gamma(\mbfs{\beta}) = |\mbfs{\beta}|$, where the absolute value is taken component-wise.
The reverse triangle inequality $||a| - |b||\leq |a - b|$ implies $\Gamma$ is non-expansive.
Set 
\(
I = \{\mbfs{e}^1, \dotsc, \mbfs{e}^m, -\mbfs{e}^1, \dotsc, -\mbfs{e}^m\} \subseteq D^m
\)
where $\mbfs{e}^i \in \mbb{Z}^m$ is the $i$th standard unit vector.
Each $\mbfs{\beta}\in D^m$ is in $\cone(J)$ for a linearly independent set $J\subseteq I$.
By definition of $I$ and because $J$ is linearly independent, each pair of distinct points $\mbfs{\beta}, \mbfs{\beta}'\in J$ is isometric. 
Theorem \ref{thmCpolyhedral} then ensures that $C_\Gamma$ is a polyhedron. 
Moreover, from Theorem \ref{thmCpolyhedral} we see that
\begin{align*}
C_\Gamma = & \{(\mbfs{x},\mbfs{y}) \in \mbb{R}^n \times \mbb{R}^m : \Gamma(\mbfs{\beta}){}^\T \mbfs{x} - \mbfs{\beta}{}^\T \mbfs{y} \ge 0 ~~ \forall ~ \mbfs{\beta} \in I\} \\ 
=  &\{(\mbfs{x},\mbfs{y}) \in \mbb{R}^n \times \mbb{R}^m : x_i \geq |y_i| ~~ \forall\ i\in \{1,\dotsc, m\}\}.
\end{align*}
For each $\mbfs{\beta} \in D^m$, the point $\Gamma(\mbfs{\beta})$ is non-negative and strictly positive in at least one component.
Hence, $(\mbfs{0}, \mbfs{0})\not\in \conv(\{(\Gamma(\mbfs{\beta}),-\mbfs{\beta})\,:\ \mbfs{\beta}\in D^m\})$.
 Theorem \ref{thm:continuous} then ensures that $C_{\Gamma}$ is a full-dimensional maximal $Q$-free set.
Figure \ref{fig:4dexample} illustrates a $3$-dimensional slice of the $4$-dimensional sets $Q$ and $C_\Gamma$ obtained for $n=m=2$.
\begin{figure}[t]
\centering
\includegraphics[scale=0.2]{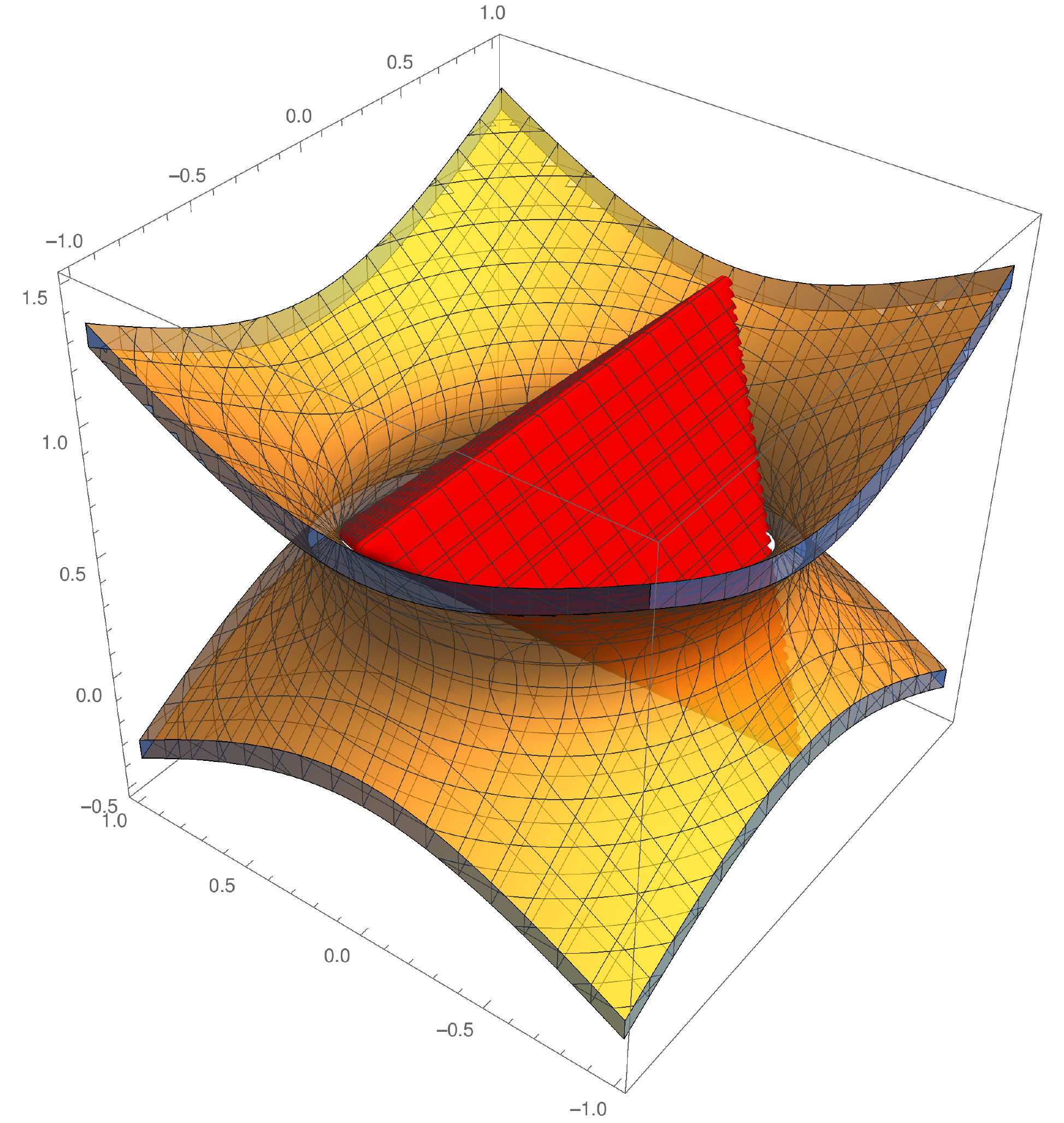}
\caption{$3$-dimensional slice of the $4$-dimensional sets $Q$ (boundary in orange) and $C_\Gamma$ (red) obtained when $n=m=2$ in Example \ref{ex2}.}
\label{fig:4dexample}
\end{figure}

The maximality of this example could not have been proved with the results of Mu\~{n}oz and Serrano~\cite{MS2022}.
Indeed, it can be seen that $C_\Gamma \cap Q = \{(\mbfs{x},\mbfs{y}) \in \mbb{R}^n \times \mbb{R}^m : x_i = |y_i| ~~ \forall\ i\in \{1,\dotsc, m\}\}$.
Therefore, every facet of $C_\Gamma$ intersects $Q$ and, more importantly, any $(\mbfs{x},\mbfs{y}) \in C_\Gamma\cap Q$ is contained in $m$ different facets of $C_\Gamma$. 
Consequently, there is no exposing point in $C_\Gamma \cap Q$ for any of the facets of $C_\Gamma$.\\
\end{example}

\begin{example}\label{ex3}
In this example we show how to construct a polyhedral maximal $Q$-free set from a starting set of points in $D^m$.
In Figure \ref{fig:6dexample} (left) consider $\{\pm\mbfs{e}^3\}$ (in blue), $\{\pm \mbfs{e}^1, -\mbfs{e}^2\}$ (in red), $\{\mbfs{e}^2\}$ (in green), and $\{-\sfrac{1}{\sqrt{2}}\cdot \mbfs{e}^1 \pm \sfrac{1}{\sqrt{2}}\cdot \mbfs{e}^2\}$ (in black) in $D^3$; let $I$ be the set of these $8$ points.
We define $\Gamma$ to map $\{\pm\mbfs{e}^3\}$ to $\mbfs{e}^3$, $\{\pm \mbfs{e}^1, -\mbfs{e}^2\}$ to $-\mbfs{e}^2$, $\{\mbfs{e}^2\}$ to $-\mbfs{e}^1$, and $\{-\sfrac{1}{\sqrt{2}}\cdot \mbfs{e}^1 \pm \sfrac{1}{\sqrt{2}}\cdot \mbfs{e}^2\}$ to $-\sfrac{1}{\sqrt{2}}\cdot \mbfs{e}^1 + \sfrac{1}{\sqrt{2}}\cdot \mbfs{e}^2$; see Figure \ref{fig:6dexample} (right). 
Each $\mbfs{\beta} \in D^3$ is a conic combination of at most three points in $I$ that are pairwise isometric. 
\begin{figure}[t]
\centering
\includegraphics[scale=0.28]{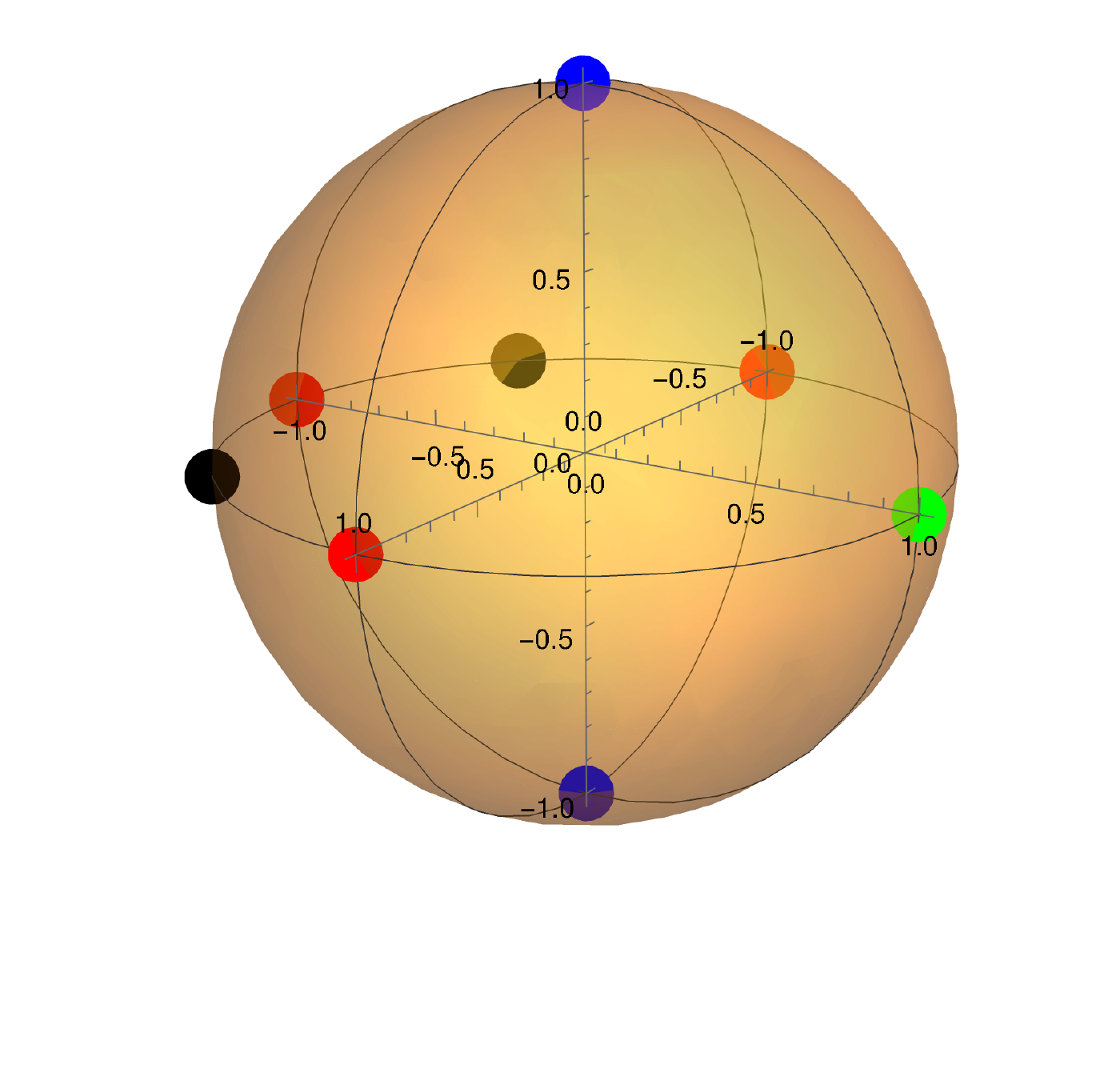}%
\includegraphics[scale=0.28]{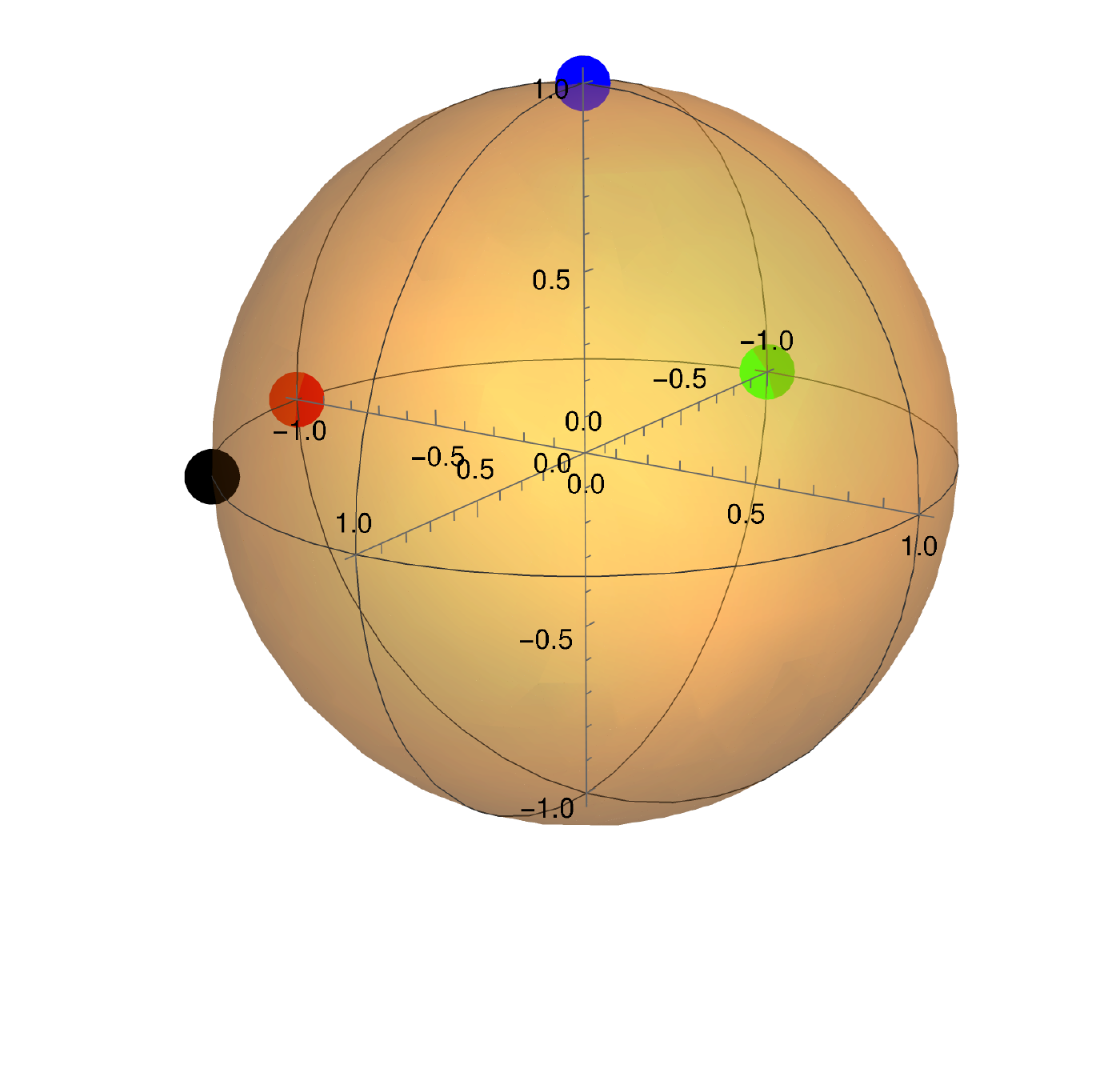}
\vspace{-0.8cm}
\caption{Representation of $\Gamma$ in the $6$-dimensional set of Example \ref{ex3}. A point on the left plot represents a $\mbfs{\beta}$ that gets mapped by $\Gamma$ to the point on the right plot of the same color.}
\label{fig:6dexample}
\end{figure}
One can extend $\Gamma$ from $I$ to a non-expansive function on $D^3$ through a conic interpolation; see Lemma~\ref{lemIsoSets}.
Note that the set $\{\Gamma(\mbfs{\beta}):\ \mbfs{\beta} \in D^3\}$ is contained in a pointed cone.
Thus, we also have $(\mbfs{0}, \mbfs{0})\not\in \conv(\{(\Gamma(\mbfs{\beta}),-\mbfs{\beta})\,:\ \mbfs{\beta}\in D^3\})$.
Theorem~\ref{thm:continuous} then implies that
\[
C_\Gamma = \{(\mbfs{x},\mbfs{y}) \in \mbb{R}^3 \times \mbb{R}^3 : \Gamma(\mbfs{\beta}){}^\T \mbfs{x} - \mbfs{\beta}{}^\T \mbfs{y} \ge 0 ~~ \forall\ \mbfs{\beta} \in I\}.
\]
is maximal $Q$-free.
Figure \ref{fig:6dexample2} shows a 3-dimensional slice of the 6-dimensional $C_\Gamma$.
\begin{figure}[t]
\centering
\includegraphics[scale=0.18]{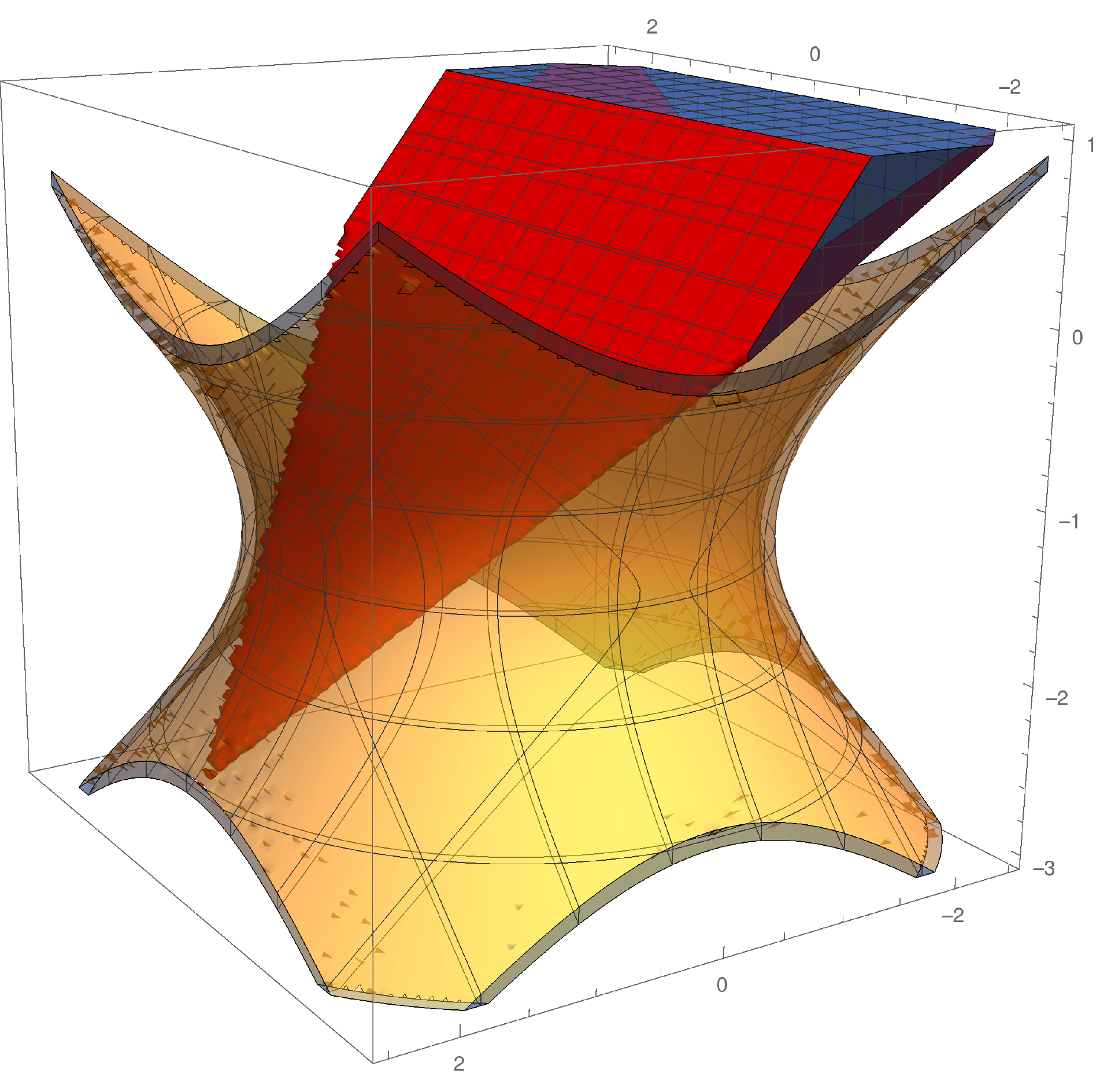}
\caption{$3$-dimensional slices of the 6-dimensional sets $Q$ (boundary in orange) and $C_\Gamma$ (red) obtained using the $\Gamma$ function depicted in Figure \ref{fig:6dexample}. We note that maximality may not be evident in this picture since maximality is not preserved when taking slices.}
\label{fig:6dexample2}
\end{figure}

The construction in this example can be generalized to an arbitrarily large set $I$ as long as their conic combinations generate $D^m$.
This produces a maximal $Q$-free polyhedra with arbitrarily many facets.
\end{example}

%%%%%%%%%%%%%%%%%%%
%%%%%%%%%%%%%%%%%%%
%%%%%%%%%%%%%%%%%%%

\section{A proof of Theorem~\ref{thmMax}}\label{secExposing}

The remainder of the paper is dedicated to proving our results. 
We begin these proof sections with the maximality criterion in Theorem~\ref{thmMax}. 
We remark that this results hold for arbitrary $S$, and not only for the quadratically-defined $Q$.\\

\proof[Proof of Theorem~\ref{thmMax}]
  $(\Leftarrow)$ Assume to the contrary that $C$ is not maximal and let $K \supsetneq C$ be a convex $S$-free set.
  There exists $(\overline{\mbfs{\alpha}}, \overline{\alpha}_0) \in I$ such that $\overline{\mbfs{\alpha}}^\T \mbfs{x} \le \overline{\alpha}_0$ is not valid for $K$.
  Let $(\mbfs{x}^t)_{t=1}^\infty$ be a sequence in $S$ as in the hypothesis of the theorem.
  Since $\mbfs{x}^t \in S$ and $K$ is $S$-free, there exists an inequality $\mbfs{\delta}^t{}^\T\mbfs{x} \le \delta^t_0$ such that $\|\mbfs{\delta}^t\| = \|\overline{\mbfs{\alpha}}\|$, is valid for $K$, and $\mbfs{\delta}^t{}^\T \mbfs{x}^t \ge {\delta^t_0}$.
  The inequality $\mbfs{\delta}^t{}^\T\mbfs{x} \le \delta^t_0$ is valid for $C$ because $C \subsetneq K$.
  By the definition of exposing sequence, we have $\lim_{t\to\infty} (\mbfs{\delta}^t, \delta^t_0)= (\overline{\mbfs{\alpha}}, \overline{\alpha}_0)$, so $\overline{\mbfs{\alpha}}{}^\T \mbfs{x}\le  \overline{\alpha}_0$ is valid for $K$.
  This is a contradiction.

\medskip
   \noindent $(\Rightarrow)$ Let $C^{\circ} \subseteq \mbb{R}^{d} \times \mbb{R}$ be the polar of $C$, i.e., the set of coefficients corresponding to valid inequalities for $C$:
    \[
  C = \{\mbfs{x}\in \mbb{R}^d : \mbfs{\alpha}{}^\T \mbfs{x} \le \alpha_0~\forall ~ (\mbfs{\alpha}, \alpha_0) \in C^{\circ}\} .
  \]

  The set $C^{\circ}$ is a closed convex cone, and $C^{\circ}$ is pointed because $C$ is full-dimensional.
  Since $C^{\circ}$ is a closed pointed cone, there is a set $I \subseteq C^{\circ}$ generating the extreme rays of $C^{\circ}$ and $C^{\circ} = \cone(I)$; see, e.g.,~\cite[Theorem 18.5]{R1970}.
  Hence,
  \[
  C 
  =\left\{\mbfs{x} \in \mbb{R}^d : \mbfs{\alpha}{}^\T \mbfs{x} \le \alpha_0~~\forall ~ (\mbfs{\alpha},\alpha_0) \in I \right\}.
  \]
  Without loss of generality, we can scale each $(\mbfs{\alpha}, \alpha_0) \in I$ by a positive number so that $\|(\mbfs{\alpha}, \alpha_0)\| = 1$.
  Let $(\overline{\mbfs{\alpha}}, \overline{\alpha}_0) \in I$ and define
  \begin{align*}
    K_t &:= \left\{(\mbfs{\alpha}, {\alpha}_0) \in \mbb{R}^d \times \mbb{R} :\ \|(\overline{\mbfs{\alpha}}, \overline{\alpha}_0) - (\mbfs{\alpha}, {\alpha}_0)\| < \frac{1}{t}\right\}, \\
    C_t &:= \left\{ \mbfs{x} \in \mbb{R}^d: \mbfs{\alpha}{}^\T \mbfs{x} \leq \alpha_0~ \forall ~(\mbfs{\alpha}, \alpha_0) \in I \setminus K_t\right\}.
  \end{align*}

  We proceed to show that $C \subsetneq C_t$.
  Assume to the contrary that $C = C_t$.
  This implies that $I \subseteq \overline{\cone}(I \setminus K_t)$.
  In particular, $(\overline{\mbfs{\alpha}}, \overline{\alpha}_0) \in \overline{\cone}(I \setminus K_t)$.
  Thus, by Carath\'eodory's Theorem, we can write
  \[
  (\overline{\mbfs{\alpha}}, \overline{\alpha}_0)  = \lim_{s \to\infty} \sum_{i=1}^{d+1} \lambda_{[i,s]}(\mbfs{\alpha}_{[i,s]}, \alpha_{[0,i,s]})
  \]
  where each $\lambda_{[i,s]} \ge 0$ and $(\mbfs{\alpha}_{[i,s]}, \alpha_{[0,i,s]}) \in I\setminus K_t$.
  Given that each point in $I$ has norm $1$ and $I \setminus K_t$ is closed, we may assume that $(\mbfs{\alpha}_{[i,s]}, \alpha_{[0,i,s]})$ converges to a norm-1 vector $(\mbfs{\alpha}_{[i]}, \alpha_{[0,i]}) \in C^{\circ}  \setminus K_t$ for each $i$.
  Furthermore, we assume that $\lambda_{[i,s]}$ converges to $\lambda_i \ge 0$ for each $i$; otherwise, say $\lambda_{[1,s]}$ diverges to $\infty$, then 
  \begin{align*}
  \mbfs{0} = \lim_{s\to\infty} \frac{1}{\lambda_{[1,s]}}(\overline{\mbfs{\alpha}}, \overline{\alpha}_0)   &= \lim_{s\to\infty}  \frac{1}{\lambda_{[1,s]}}\sum_{i=1}^{d+1} \lambda_{[i,s]}(\mbfs{\alpha}_{[i,s]}, \alpha_{[0,i,s]})\\
  & = (\mbfs{\alpha}_{[1]}, \alpha_{[0,1]}) + \lim_{s\to\infty}  \sum_{i=2}^{d+1} \frac{\lambda_{[i,s]}}{\lambda_{[1,s]}}(\mbfs{\alpha}_{[i,s]}, \alpha_{[0,i,s]}).
  \end{align*}
  However, then $ \lim_{s\to\infty}  \sum_{i=2}^{d+1} \frac{\lambda_{[i,s]}}{\lambda_{[1,s]}}(\mbfs{\alpha}_{[i,s]}, \alpha_{[0,i,s]}) = -(\mbfs{\alpha}_{[1]}, \alpha_{[0,1]})$, implying that $C^{\circ}$ is not pointed, which is a contradiction.
  Hence,  
  \[
  (\overline{\mbfs{\alpha}}, \overline{\alpha}_0) = \lim_{s \to\infty} \sum_{i=1}^{d+1} \lambda_{[i,s]}(\mbfs{\alpha}_{[i,s]}, \alpha_{[0,i,s]}) =  \sum_{i=1}^{d+1} \lambda_{i}(\mbfs{\alpha}_{i}, \alpha_{[0,i]}),
  \]
   where each $(\mbfs{\alpha}_{i}, \alpha_{[0,i]})$ is in $C^{\circ} \setminus K_t$. 
    However, since $K_t$ is a ball of positive radius, this means that $(\overline{\mbfs{\alpha}}, \overline{\alpha}_0)$ does not generate an extreme ray of $C^{\circ}$. 
  Therefore, $C \subsetneq C_t$.

  Due to the maximality of $C$, there is a vector $\mbfs{x}^t \in \intr(C_t) \cap S$.
  If $\mbfs{\delta}^t{}^\T \mbfs{x}\le\delta^t_0$ is valid for $C$ and ${\mbfs{\delta}^t}{}^\T \mbfs{x}^t \geq {\delta^t_0}$, then  $(\mbfs{\delta}^t, {\delta^t_0}) \in K_t$.
  Thus, $\lim_{t\to \infty}(\mbfs{\delta}^t, {\delta^t}) = (\overline{\mbfs{\alpha}}, \overline{\alpha}_0)$.
  \endproof
  
%%%%%%%%%%%%%%%%%%%
%%%%%%%%%%%%%%%%%%%
%%%%%%%%%%%%%%%%%%%

\section{A proof of Theorem~\ref{thmStdForm}}\label{secHomogeneous}
We use the following theorem from Rockafellar~\cite{R1970}, albeit slightly rephrased.

 \begin{theorem}[Theorem 17.3 from~\cite{R1970}]\label{Rockafellarpatched}
Let $S^* \subseteq \mathbb{R}^{d}\times \mbb{R}$ be a non-empty compact set of vectors $(\mbfs{x}^*,\mu^*)\in \mathbb{R}^{d}\times \mbb{R}$, and set 
\[
B := \left\{\mbfs{x} \in \mbb{R}^d\,:\, \mbfs{x}{}^\top \mbfs{x}^* \leq \mu^* \quad \forall\ (\mbfs{x}^*,\mu^*)\in S^*\right\}.
\]
Suppose $B$ is full-dimensional and $(\mbfs{0}, 0)\not\in S^*$. 
Then, for a given vector $(\widehat{\mbfs{x}},\widehat{\mu})$, where $\widehat{\mbfs{x}}\neq \mbfs{0}$, the inequality
$\mbfs{x}{}^\top \widehat{\mbfs{x}} \leq \widehat{\mu}$ is valid for $B$ if and only if there exist vectors $(\mbfs{x}^*_1, \mu^*_1), \dotsc, (\mbfs{x}^*_k, \mu^*_k)\in S^*$ with $k \leq d$, and $\lambda_1, \dotsc, \lambda_k \geq 0$ such that
\[
\widehat{\mbfs{x}} = \sum_{i=1}^k \lambda_i \mbfs{x}^*_i\quad\text{and}\quad \sum_{i=1}^k \lambda_i \mu^*_i \leq \mu^*.
\]
\end{theorem}

\begin{remark}\label{rmk:rockafellar}
Theorem \ref{Rockafellarpatched} has a slight difference with the statement of Theorem 17.3 from~\cite{R1970}: the latter is missing the assumption $(\mbfs{0}, 0)\not\in S^*$. This assumption is a subtle detail which is implicitly used in the proof and in all the discussions leading to Theorem 17.3. Moreover, one can prove that without it, the theorem does not hold.

Consider $d=2$ and $S^* = \{(\cos(\theta)-1, \sin(\theta),0)\,:\, \theta\in [0,\pi]\}$, which is compact. Note that $\theta=0$ implies $(\mbfs{0},0) \in S^* $. It can be shown that in this case $B=\{(x_1,x_2)\, :\, x_1\geq 0,\, x_2\leq 0\}$ and thus $B$ is full-dimensional. The valid inequality $x_2 \leq 0$ can be obtained by noting that 
\[\frac{(\cos(\theta)-1)}{\theta} x_1 + \frac{\sin(\theta)}{\theta}x_2 \leq 0 \]
is always valid for $\theta > 0$, and then taking limit $\theta\to 0$. However, $(0,1)$ cannot be obtained as a conic combination of vectors of the form $(\cos(\theta)-1, \sin(\theta))$ $\theta\in [0,\pi]$. 
\end{remark}
\begin{remark}\label{rmk:rockafellar2}
We also note that in Theorem \ref{Rockafellarpatched} we can replace the requirements ``$B$ is full-dimensional and $(\mbfs{0}, 0)\not\in S^*$'' simply with ``$(\mbfs{0}, 0)\not\in \conv(S^*)$''. The proof is exactly the same, as the first two assumptions are used to prove the second. We use of this subtle variation in one proof below.
\end{remark}

To prove Theorem~\ref{thmStdForm}, we write $Q= \bigcup_{\mbfs{\beta} \in D^m} Q_{\mbfs{\beta}}$, where
\begin{align}
Q_{\mbfs{\beta}} :\, =& \left\{(\mbfs{x},\mbfs{y}) \in \mbb{R}^n \times \mbb{R}^m : \|\mbfs{x}\| - \mbfs{\beta}{}^\T \mbfs{y} \le 0\right\} \nonumber \\ 
 =& \left\{(\mbfs{x},\mbfs{y}) \in \mbb{R}^n \times \mbb{R}^m : \mbfs{\gamma}{}^\T \mbfs{x} - \mbfs{\beta}{}^\T \mbfs{y} \le 0~~\forall \ \mbfs{\gamma} \in D^n\right\}. \label{eq:Qbetadescription}
\end{align}
Note that each $Q_{\mbfs{\beta}}$ is convex.
The following lemma is a direct consequence of Theorem~\ref{Rockafellarpatched}.

\begin{lemma}\label{lemPolar}
  Let $\mbfs{\beta} \in D^m$.
  Every tight valid inequality for $Q_{\mbfs{\beta}}$ has the form (possibly after scaling by a positive number) $\mbfs{\gamma}{}^\T \mbfs{x} - \mbfs{\beta}{}^\T \mbfs{y} \le 0 $ for some $\|\mbfs{\gamma}\| \le 1$.
%  %
\end{lemma}

We remark that this lemma does not follow immediately from \eqref{eq:Qbetadescription}; 
Lemma \ref{lemPolar} refers to \emph{every} tight valid inequality, which, in principle, can include inequalities
not explicitly considered in the description \eqref{eq:Qbetadescription}.

We now proceed to the main proof of this section.

\begin{proof}[of Theorem~\ref{thmStdForm}]
  For each $\mbfs{\beta} \in D^m$, there is a hyperplane separating $C$ and $Q_{\mbfs{\beta}}$ because both sets are
  convex and $C$ is $Q$-free.
  For each $\mbfs{\beta} \in D^m$, by Lemma~\ref{lemPolar} we can take the corresponding inequality to be $ \mbfs{\gamma}^\T \mbfs{x} - \mbfs{\beta}^\T \mbfs{y} \ge 0$ with $\|\mbfs{\gamma}\| \leq 1$.
From this discussion, it follows that we can define a function $\gamma : D^m \to \{\mbfs{x}\in \mbb{R}^n: \|\mbfs{x}\|\le 1\}$ such that  for each $\mbfs{\beta}$, we have that $\gamma(\mbfs{\beta})^\T \mbfs{x} - \mbfs{\beta}^\T \mbfs{y} \ge 0$ is valid for $C$ and separates $Q_{\mbfs{\beta}}$ from $C$.
  Thus, 
  \begin{equation}\label{eq:inclusion1}
    C \subseteq  \left\{(\mbfs{x},\mbfs{y}) \in \mbb{R}^n \times \mbb{R}^m : \gamma(\mbfs{\beta})^\T \mbfs{x} - \mbfs{\beta} ^\T \mbfs{y} \ge 0 ~~ \forall ~ \mbfs{\beta} \in D^m\right\}.
  \end{equation}
  Each $Q_{\mbfs{\beta}}$ is separated from the set on the right-hand side of \eqref{eq:inclusion1}, implying that it is $Q$-free. 
  By the maximality of $C$, we have that \eqref{eq:inclusion1} is an equality.

  We now show that we can further restrict $\gamma(\mbfs{\beta})$ to have unit norm since $C$ is a maximal $Q$-free set\footnote{Note that Lemma \ref{lemPolar} does not directly imply that $\gamma(\mbfs{\beta})$ can be assumed to have unit norm. Moreover, one could produce a (not necessarily maximal) $Q$-free set with $\gamma$ that satisfies $\|\gamma(\mbfs{\beta})\| < 1$ for some $\mbfs{\beta}$.}.
  For this, consider the pair of valid inequalities for $C$:
 \[
    \gamma(\mbfs{\beta})^\T \mbfs{x} - \mbfs{\beta}^\T \mbfs{y} \ge 0 \qquad\text{and}\qquad
    \gamma(-\mbfs{\beta})^\T \mbfs{x} - (-\mbfs{\beta})^\T \mbfs{y} \ge 0 
  \]
  for each $\mbfs{\beta} \in D^m$.
  Multiplying the first inequality by $\lambda + 1$, the second one by $\lambda$ with $\lambda \geq 0$, and adding them we obtain the following valid inequality for $C$:
  \[
   \big(\lambda(\gamma(\mbfs{\beta}) + \gamma(-\mbfs{\beta})) + \gamma(\mbfs{\beta})\big){}^\T \mbfs{x} - \mbfs{\beta}^\T \mbfs{y} \ge 0.
  \]
  Notice that $\gamma(\mbfs{\beta}) + \gamma(-\mbfs{\beta}) \neq \mbfs{0}$ for each $\mbfs{\beta}$ as otherwise $-( \gamma(-\mbfs{\beta})^\T \mbfs{x} - (-\mbfs{\beta})^\T \mbfs{y} ) =  \gamma(\mbfs{\beta})^\T \mbfs{x} - \mbfs{\beta}^\T \mbfs{y} \ge 0$ is valid for $C$ implying that $C$ satisfies an equation and is not full-dimensional; this is a contradiction. 
So, there exists $\lambda(\mbfs{\beta}) \geq 0$ such that $\|\lambda(\mbfs{\beta}) (\gamma(\mbfs{\beta}) + \gamma(-\mbfs{\beta})) + \gamma(\mbfs{\beta})\| = 1$. 
  Define $\Gamma : D^m \to D^n$ by
  \[
    \Gamma(\mbfs{\beta}) := \lambda(\mbfs{\beta}) (\gamma(\mbfs{\beta}) + \gamma(-\mbfs{\beta})) + \gamma(\mbfs{\beta}).
  \]
  Therefore,
  \begin{equation} \label{eq:inclusion}
    C \subseteq  \left\{(\mbfs{x},\mbfs{y}) \in \mbb{R}^n \times \mbb{R}^m : \Gamma(\mbfs{\beta})^\T \mbfs{x} - \mbfs{\beta} ^\T \mbfs{y} \ge 0~~ \forall ~ \mbfs{\beta} \in D^m\right\} = C_\Gamma.
  \end{equation}
  Again using that the right-hand side in \eqref{eq:inclusion} is $Q$-free and $C$ is maximal, we conclude that \eqref{eq:inclusion} is an equality.  
\end{proof}

%%%%%%%%%%%%%%%%%%%%%%%%
%%%%%%%%%%%%%%%%%%%%%%%%

\section{Preliminary lemmata}\label{secNEF}

In this section we collect a variety of lemmata to prove our characterization of maximality (Theorem~\ref{thm:continuous}). 
Each lemma is classified as a result from convex analysis or about non-expansive functions.

\subsection{Preliminaries about convexity}

Many of the results in this subsection follow from known results in convexity theory.
When proving these results, we often use the fact that if $\Gamma$ is continuous, then $\{(\Gamma(\mbfs{\beta}),-\mbfs{\beta})\ : \mbfs{\beta} \in D^m\}$ is compact.
Consequently, 
\[
\conv(\{(\Gamma(\mbfs{\beta}),-\mbfs{\beta})\ : \mbfs{\beta} \in D^m\}) = \overline{\conv}(\{(\Gamma(\mbfs{\beta}),-\mbfs{\beta})\ : \mbfs{\beta} \in D^m\}).
\]
See~\cite[Theorem 17.2]{R1970}.
We begin by providing conditions under which $C_\Gamma$ is full-dimensional. In what follows, we use $C_{\Gamma}^{\circ}$ to denote the reverse polar cone of $C_\Gamma$, i.e., the set of coefficients corresponding to tight valid inequalities for $C_{\Gamma}$:
\[
C^{\circ}_{\Gamma} := \left\{\mbfs{c}\in \mbb{R}^d:\ \mbfs{c}^\top\mbfs{x} \ge 0 \quad \forall ~ \mbfs{x} \in C_{\Gamma}\right\}.
\]

\begin{lemma}\label{lemNoZeroSameCoeff}
Let $\Gamma:D^m \to D^n$ be continuous and define $C_{\Gamma}$ as in~\eqref{Cgammadef}.
Then $C_{\Gamma}$ is full-dimensional if and only if $(\mbfs{0}, \mbfs{0}) \not\in \conv(\{(\Gamma(\mbfs{\beta}),-\mbfs{\beta})\ : \mbfs{\beta} \in D^m\})$.
\end{lemma}

\begin{proof}
($\Rightarrow$) Assume that $C_{\Gamma}$ is full-dimensional.
Furthermore, assume to the contrary that $(\mbfs{0}, \mbfs{0}) \in \conv \{(\Gamma(\mbfs{\beta}),-\mbfs{\beta})\ : \mbfs{\beta} \in D^m\}$.
The set $\{(\Gamma(\mbfs{\beta}),-\mbfs{\beta})\ : \mbfs{\beta} \in D^m\}$ is compact because $\Gamma$ is continuous, so its closed convex hull is equal to its convex hull.
Hence, there exists a set $\{\mbfs{\beta}^i\}_{i=1}^k \subseteq D^m$ and numbers $\lambda_1, \dotsc, \lambda_k > 0$ such that
\[
\sum_{i=1}^k \lambda_i (\Gamma(\mbfs{\beta}^i), - \mbfs{\beta}^i) = (\mbfs{0}, \mbfs{0}). 
\]
Note that $k>1$, as $\mbfs{\beta}\neq \mbfs{0}$ for all $\mbfs{\beta}\in D^m$.  This implies
\[
\mbfs{\beta}^k = - \sum_{i=1}^{k-1} \frac{\lambda_i}{\lambda_k} \mbfs{\beta}^i 
\quad \text{and} \quad
\Gamma(\mbfs{\beta}^k) = - \sum_{i=1}^{k-1} \frac{\lambda_i}{\lambda_k} \Gamma(\mbfs{\beta}^i).
\]
The valid inequality $\Gamma(\mbfs{\beta}^k)^\top \mbfs{x} - \mbfs{\beta}^k \mbfs{y} \geq 0$ can then be rewritten as
\[
-\left(\sum_{i=1}^{k-1} \frac{\lambda_i}{\lambda_k} \Gamma(\mbfs{\beta}^i) \right){}^\T \mbfs{x} \geq -\left(\sum_{i=1}^{k-1} \frac{\lambda_i}{\lambda_k} \mbfs{\beta}^i\right){}^\T \mbfs{y}.
\]
On the other hand, taking a conic combination of the inequalities given by $\{\mbfs{\beta}^1, \dotsc, \mbfs{\beta}^{k-1}\}$ using the weights $\sfrac{\lambda_i}{\lambda_k} > 0$ we obtain the valid inequality
\[
\left(\sum_{i=1}^{k-1} \frac{\lambda_i}{\lambda_k} \Gamma(\mbfs{\beta}^i) \right){}^\T \mbfs{x} \ge \left(\sum_{i=1}^{k-1} \frac{\lambda_i}{\lambda_k} \mbfs{\beta}^i\right){}^\T \mbfs{y}.
\]
Hence, $C_\Gamma \subseteq \{(\mbfs{x},\mbfs{y}) \, :\, \Gamma(\mbfs{\beta}^k)^\top \mbfs{x} - \mbfs{\beta}^k \mbfs{y} = 0 \}$. 
Since $(\Gamma(\mbfs{\beta}^k), - \mbfs{\beta}^k) \neq (\mbfs{0}, \mbfs{0})$, we conclude that $C_\Gamma$ is contained in a hyperplane. 
This is a contradiction.

\medskip

($\Leftarrow$) Assume that $(\mbfs{0}, \mbfs{0}) \not\in \conv(\{(\Gamma(\mbfs{\beta}),-\mbfs{\beta})\ : \mbfs{\beta} \in D^m\})$.
Assume to the contrary that $C_{\Gamma}$ is not full-dimensional.
Hence, there exists a non-zero vector $(\mbfs{g}, \mbfs{b} ) \in \mbb{R}^n\times \mbb{R}^m$ such that $\mbfs{g}^\top\mbfs{x}- \mbfs{b}^\top \mbfs{y} = 0$ for all $(\mbfs{x}, \mbfs{y}) \in C_{\Gamma}$; note that the right-hand side is $0$ because $(\mbfs{0}, \mbfs{0}) \in C_{\Gamma}$.
In other words, both $\mbfs{g}^\top\mbfs{x}- \mbfs{b}^\top \mbfs{y} \ge 0$ and $-\mbfs{g}^\top\mbfs{x}- (-\mbfs{b})^\top \mbfs{y} \ge 0$ are valid for $C_{\Gamma}$, so $(\mbfs{g}, -\mbfs{b}), (-\mbfs{g}, \mbfs{b}) \in C^{\circ}_{\Gamma}$.
Since $(\mbfs{0}, \mbfs{0}) \not\in \conv(\{(\Gamma(\mbfs{\beta}),-\mbfs{\beta})\ : \mbfs{\beta} \in D^m\})$, both $(\mbfs{g}, -\mbfs{b})$ and $(-\mbfs{g}, \mbfs{b}) $ are in $\cone(\{(\Gamma(\mbfs{\beta}),-\mbfs{\beta})\ : \mbfs{\beta} \in D^m\})$ (see Remark \ref{rmk:rockafellar2}).
So, $(\mbfs{g}, -\mbfs{b})$ and $(-\mbfs{g}, \mbfs{b}) $ are conic combinations of vectors in $\{(\Gamma(\mbfs{\beta}),-\mbfs{\beta})\ : \mbfs{\beta} \in D^m\}$.
However, after properly scaling $(\mbfs{g}, -\mbfs{b})$ and $(-\mbfs{g}, \mbfs{b}) $ by positive numbers, we see that $(\mbfs{0}, \mbfs{0}) \in \conv(\{(\Gamma(\mbfs{\beta}),-\mbfs{\beta})\ : \mbfs{\beta} \in D^m\})$, which is a contradiction.
\end{proof}

The next lemma provides a way of describing all the valid inequalities of $C_\Gamma$ when $\Gamma$ is continuous.
\begin{lemma}\label{lemCCinequalities}
Let $\Gamma:D^m \to D^n$ be continuous and define $C_\Gamma$ as in \eqref{Cgammadef}.
Assume that $(\mbfs{0}, \mbfs{0}) \not\in \conv(\{(\Gamma(\mbfs{\beta}),-\mbfs{\beta})\ : \mbfs{\beta} \in D^m\})$.
If $\overline{\mbfs{\gamma}}{}^\T \mbfs{x} - \overline{\mbfs{\beta}}{}^\T \mbfs{y} \ge 0$ is valid for $C_\Gamma$, then there exist $\mbfs{\beta}^1, \dotsc, \mbfs{\beta}^k$ such that $(\overline{\mbfs{\gamma}\vphantom{\beta}}, -\overline{\mbfs{\beta}}) \in \cone(\{(\Gamma(\mbfs{\beta})^i, -\mbfs{\beta}^i):\ i \in \{1, \dotsc, k\}\})$. 
\end{lemma}
We omit the proof of Lemma~\ref{lemCCinequalities} as it follows directly from Lemma \ref{lemNoZeroSameCoeff} and Theorem~\ref{Rockafellarpatched}.

The following lemma gives the representation of $C_{\Gamma}$ that we will use to prove maximality in Subsection~\ref{secMainThmNecessary}.
For a pointed closed convex cone $K\subseteq \mbb{R}^d$, we say that $\mbfs{v} \in K$ is an {\it exposed ray} if $\cone(\mbfs{v})$ is an exposed ray of $K$. 

\begin{lemma}\label{lemPointed}
Let $\Gamma:D^m \to D^n$ be continuous and define $C_{\Gamma}$ as in~\eqref{Cgammadef}.
Assume that
\(
(\mbfs{0}, \mbfs{0}) \not\in \conv(\{(\Gamma(\mbfs{\beta}),-\mbfs{\beta})\ : \mbfs{\beta} \in D^m\}).
\)
Then
\begin{align*}
  C_{\Gamma} &= \left\{(\mbfs{x}, \mbfs{y}) \in \mbb{R}^n\times \mbb{R}^m:\ \Gamma(\mbfs{\beta})^\top \mbfs{x} -  \mbfs{\beta}{}^\top \mbfs{y}\ge 0 \quad\forall ~ (\Gamma(\mbfs{\beta}),-\mbfs{\beta}) \in K\right\}\\
  &= \left\{(\mbfs{x}, \mbfs{y}) \in \mbb{R}^n\times \mbb{R}^m:\ \Gamma(\mbfs{\beta})^\top \mbfs{x} -  \mbfs{\beta}{}^\top \mbfs{y}\ge 0 \quad\forall ~ (\Gamma(\mbfs{\beta}),-\mbfs{\beta}) \in I\right\}.
\end{align*}
where $I$ is the set of exposed rays of $K:= \cone(\left\{(\Gamma(\mbfs{\beta}),-\mbfs{\beta})\,: \mbfs{\beta} \in D^m\,\right\})$.
\end{lemma}
\proof
The first equation follows from Lemma~\ref{lemCCinequalities}.
To prove the second equation it is sufficient to show that $K = \overline{\cone}(I)$.
Define
\(
H := \left\{(\Gamma(\mbfs{\beta}),-\mbfs{\beta})\,: \mbfs{\beta} \in D^m\,\right\}.
\)
Notice that $K = \cone(H) = \cone(\conv(H))  = C^{\circ}_{\Gamma}$, where the second equation follows by definition and the third equation follows from Lemma~\ref{lemCCinequalities}.
Observe that $K$ is pointed because $(\mbfs{0}, \mbfs{0}) \not \in \conv(H)$.
Furthermore, $K$ is closed because $C^{\circ}_{\Gamma}$ is closed.
Given that $K$ is a closed and pointed convex cone, it follows that $K = \overline{\cone}(I)$; see~\cite[Theorems 18.5 and 18.6]{R1970}.
\endproof

\subsection{Preliminaries about non-expansive functions}

\begin{lemma}\label{lemIsoSets}
Let $\Gamma:D^m \to D^n$ be non-expansive.
Let $I \subseteq D^m$ be a finite set of pairwise isometric points. 
The following properties hold true:
\begin{enumerate}[leftmargin = *]
\item\label{propLin1} If $\sum_{\mbfs{\beta} \in I}\epsilon_{\mbfs{\beta}}\mbfs{\beta} \in D^m$, where $\epsilon_{\mbfs{\beta}} \ge 0$ for each $\mbfs{\beta} \in I$, then  $\sum_{\mbfs{\beta} \in I}\epsilon_{\mbfs{\beta}}\Gamma(\mbfs{\beta})\in D^n$ and $\Gamma(\sum_{\mbfs{\beta} \in I}\epsilon_{\mbfs{\beta}}\mbfs{\beta}) = \sum_{\mbfs{\beta} \in I}\epsilon_{\mbfs{\beta}}\Gamma(\mbfs{\beta} )$.
\item\label{propLin2} If $\sum_{\mbfs{\beta} \in I}\epsilon_{\mbfs{\beta}}\Gamma(\mbfs{\beta}) \in D^n$, where $\epsilon_{\mbfs{\beta}} \ge 0$ for each $\mbfs{\beta} \in I$, then $\sum_{\mbfs{\beta} \in I}\epsilon_{\mbfs{\beta}}\mbfs{\beta}\in D^m$ and $\Gamma(\sum_{\mbfs{\beta} \in I}\epsilon_{\mbfs{\beta}}\mbfs{\beta}) = \sum_{\mbfs{\beta} \in I}\epsilon_{\mbfs{\beta}}\Gamma(\mbfs{\beta} )$.

\end{enumerate}
\end{lemma}

\begin{proof}
Set $\widehat{\mbfs{\beta}} := \sum_{\mbfs{\beta} \in I}\epsilon_{\mbfs{\beta}}\mbfs{\beta} $.
Using the isometry of points in $I$, we have
\[
 \|\sum_{\mbfs{\beta} \in I} \epsilon_{\mbfs{\beta}} \Gamma(\mbfs{\beta}) \|^2 
 = \sum_{\mbfs{\beta}, \mbfs{\beta}' \in I} \epsilon_{\mbfs{\beta}}\epsilon_{\mbfs{\beta}'} \Gamma(\mbfs{\beta}){}^\T \Gamma(\mbfs{\beta}') \\[.1 cm]
 = \sum_{\mbfs{\beta}, \mbfs{\beta}' \in I} \epsilon_{\mbfs{\beta}}\epsilon_{\mbfs{\beta}'} \mbfs{\beta}^\T \mbfs{\beta}' 
 =  \|\widehat{\mbfs{\beta}}\|^2. 
\]
In the case of Property{\em~\ref{propLin1}}, we assume $ \widehat{\mbfs{\beta}} \in D^m$, so the previous equation proves that $ \|\sum_{\mbfs{\beta} \in I} \epsilon_{\mbfs{\beta}} \Gamma(\mbfs{\beta}) \| = 1$.
In the case of Property{\em~\ref{propLin2}}, we assume $ \|\sum_{\mbfs{\beta} \in I} \epsilon_{\mbfs{\beta}} \Gamma(\mbfs{\beta}) \|^2 = 1$, so the previous equation proves that $\widehat{\mbfs{\beta}} \in D^m$. 
Therefore, it remains to show that
$\Gamma(\widehat{\mbfs{\beta}} ) = \sum_{\mbfs{\beta} \in I}\epsilon_{\mbfs{\beta}}\Gamma(\mbfs{\beta})$ in both cases; we
prove these simultaneously.

Using the non-expansive property of $\Gamma$ and the nonnegativity of $\epsilon_{\mbfs{\beta}}$, we have
\[
1 = \widehat{\mbfs{\beta}}{}^\T \widehat{\mbfs{\beta}}  
= \sum_{\mbfs{\beta} \in I} \epsilon_{\mbfs{\beta}} \widehat{\mbfs{\beta}}{}^\T \mbfs{\beta}
\le  \sum_{\mbfs{\beta} \in I} \epsilon_{\mbfs{\beta}} \Gamma(\widehat{\mbfs{\beta}})^\T \Gamma(\mbfs{\beta})
= 
\Gamma(\widehat{\mbfs{\beta}})^\T \big(\sum_{\mbfs{\beta} \in I} \epsilon_{\mbfs{\beta}} \Gamma(\mbfs{\beta})\big) .
\]
The Cauchy-Schwarz inequality implies that $1 = \Gamma(\widehat{\mbfs{\beta}})^\T (\sum_{\mbfs{\beta} \in I} \epsilon_{\mbfs{\beta}} \Gamma(\mbfs{\beta}))$.
Since both vectors have unit norm, we conclude that $\Gamma(\widehat{\mbfs{\beta}} ) = \sum_{\mbfs{\beta} \in I}\epsilon_{\mbfs{\beta}}\Gamma(\mbfs{\beta})$.
\end{proof}

Our final lemma states that if an inequality $\Gamma(\overline{\mbfs{\beta}})^\T\mbfs{x} - \overline{\mbfs{\beta}}{}^\T \mbfs{y} \ge 0$ is implied by other inequalities of the same form indexed by $I \subseteq D^m$, then $\overline{\mbfs{\beta}}$ must be isometric with $\mbfs{\beta} \in I$. 
This will be used in the proof of Theorem~\ref{thmCpolyhedral} to help establish that we have a covering of $D^m$ by isometric points.
\begin{lemma}\label{lempolyhedraIsometricFacets}
Let $\Gamma:D^m \to D^n$ be non-expansive.
Let $\overline{\mbfs{\beta}} \in D^m$, let $I \subseteq D^m$ be a finite set, and let $\lambda_{\mbfs{\beta}} > 0$ for each $\mbfs{\beta} \in I$ be such that $(\Gamma(\overline{\mbfs{\beta}}),\overline{\mbfs{\beta}}) = \sum_{\mbfs{\beta} \in I} \lambda_{\mbfs{\beta}}(\Gamma(\mbfs{\beta}), \mbfs{\beta})$.
Then $\overline{\mbfs{\beta}}$ and $\mbfs{\beta}$ are isometric for each $\mbfs{\beta} \in I$.
\end{lemma}

\begin{proof}
Notice that
\[
\begin{array}{rclclcl}
0 &=&\displaystyle\Gamma(\overline{\mbfs{\beta}})^\T \Gamma(\overline{\mbfs{\beta}}) - \overline{\mbfs{\beta}}{}^\T\overline{\mbfs{\beta}} &=& \displaystyle \left(\sum_{\mbfs{\beta}\in I} \lambda_{\mbfs{\beta}} \Gamma(\mbfs{\beta})\right){}^\T \Gamma(\overline{\mbfs{\beta}}) - \left(\sum_{\mbfs{\beta}\in I} \lambda_{\mbfs{\beta}} \mbfs{\beta} \right){}^\T \overline{\mbfs{\beta}}\\[.75 cm]
 && &= &  \displaystyle\sum_{\mbfs{\beta} \in I} \lambda_{\mbfs{\beta}} \left(\Gamma(\mbfs{\beta})^\T\Gamma(\overline{\mbfs{\beta}}) - {\mbfs{\beta}}{}^\T\overline{\mbfs{\beta}}\right),
\end{array}
\]
where the first equality follow from $\overline{\mbfs{\beta}} \in D^m$ and $\Gamma(\overline{\mbfs{\beta}}) \in D^n$.
Due to the non-expansiveness of $\Gamma$, every summand is non-negative.
Since the sum is 0, every summand must be 0.
As $\lambda_{\mbfs{\beta}} >0$, we conclude that $\Gamma(\mbfs{\beta})^\T\Gamma(\overline{\mbfs{\beta}}) = {\mbfs{\beta}}{}^\T\overline{\mbfs{\beta}}$ for all $\mbfs{\beta} \in I$.
\end{proof}

%%%%%%%%%%%%%%%%%%%%%%%%
%%%%%%%%%%%%%%%%%%%%%%%%

\section{A Proof of Theorem~\ref{thm:continuous}}\label{secMainThm}

For the sake of presentation, we divide the proof of Theorem~\ref{thm:continuous} into its sufficient and necessary conditions for maximality. 
The proof follows by combining Lemmata~\ref{lemNMax} and~\ref{lemSMax}.

\subsection{A sufficient condition for maximality}\label{secMainThmNecessary}

In this section, we prove the following result.
We do not explicitly assume continuity in Lemma~\ref{lemNMax} as we do in Theorem~\ref{thm:continuous} because it is implied by the non-expansivity of $\Gamma$.

\begin{lemma}\label{lemNMax}
Let $\Gamma: D^m \to D^n$ and define $C_\Gamma$ as in \eqref{Cgammadef}. 
If $\Gamma$ is non-expansive and $(\mbfs{0}, \mbfs{0}) \not\in \conv(\{(\Gamma(\mbfs{\beta}),-\mbfs{\beta})\ : \mbfs{\beta} \in D^m\})$, then $C_{\Gamma}$ is a full-dimensional maximal $Q$-free set.
\end{lemma}

\begin{proof}
The function $\Gamma$ is continuous because it is non-expansive. 
Therefore, Lemma~\ref{lemNoZeroSameCoeff} implies that $C_\Gamma$ is full-dimensional.
For the remainder of the proof, we focus on establishing that $C_{\Gamma}$ is maximal $Q$-free.

According to Lemma~\ref{lemPointed}, we can write 
\[
C_{\Gamma} = \left\{(\mbfs{x}, \mbfs{y}) \in \mbb{R}^n\times \mbb{R}^m:\ \Gamma(\mbfs{\beta})^\top \mbfs{x} -  \mbfs{\beta}{}^\top \mbfs{y}\ge 0 \quad\forall ~ (\Gamma(\mbfs{\beta}),-\mbfs{\beta}) \in I\right\}.
\]
where $I$ is the set of exposed rays of $\cone(\conv(\left\{(\Gamma(\mbfs{\beta}),-\mbfs{\beta})\,: \mbfs{\beta} \in D^m\,\right\}))$.

To prove that $C_{\Gamma}$ is maximal, it suffices to prove that the inequality $\Gamma(\mbfs{\beta})^\top\mbfs{x} - \mbfs{\beta}{}^\top\mbfs{y} \ge 0$ for each $(\Gamma(\mbfs{\beta}), -\mbfs{\beta}) \in I$ has an exposing sequence.
To this end, fix $(\Gamma(\overline{\mbfs{\beta}}), -\overline{\mbfs{\beta}}) \in I$.
Given that $(\Gamma(\overline{\mbfs{\beta}}), -\overline{\mbfs{\beta}})$ is an exposed ray of $K$, there exists a vector $(\overline{\mbfs{x}}, \overline{\mbfs{y}}) \in \mbb{R}^n \times \mbb{R}^m$ such that
\begin{equation}\label{expI}
\Gamma(\overline{\mbfs{\beta}})^\top\overline{\mbfs{x}} - \overline{\mbfs{\beta}}{}{}^\top\overline{\mbfs{y}} = 0 < \Gamma({\mbfs{\beta}})^\top\overline{\mbfs{x}} - {\mbfs{\beta}}{}{}^\top\overline{\mbfs{y}} \qquad \forall \ \mbfs{\beta} \in D^m \setminus \{\overline{\mbfs{\beta}}\}.
\end{equation}
We consider two cases; Case 2 is the more involved case that constitutes most of the remainder of the proof.

\medskip

\noindent {\bf Case 1.} Assume that $(\overline{\mbfs{x}}, \overline{\mbfs{y}}) \in Q$.
Consider the \emph{constant} sequence $((\overline{\mbfs{x}}, \overline{\mbfs{y}}) )_{t=1}^{\infty}$ of points in $Q$.\footnote{Using the language in~\cite{MS2022}, the constant sequence here is an exposing point.}
Consider any sequence $((\mbfs{g}^t, \mbfs{b}^t))_{t=1}^{\infty}$ of points in $\mbb{R}^n \times \mbb{R}^m$ with $\|(\mbfs{g}^t, \mbfs{b}^t)\| = \|(\Gamma(\overline{\mbfs{\beta}}),\overline{\mbfs{\beta}})\|$ such that the inequality $\mbfs{g}^t{}^\top \mbfs{x} - \mbfs{b}^t{}^\top \mbfs{y} \ge 0$ is valid for $C_{\Gamma}$ and separates $(\overline{\mbfs{x}}, \overline{\mbfs{y}})$, that is, $\mbfs{g}^t{}^\top \overline{\mbfs{x}} - \mbfs{b}^t{}^\top \overline{\mbfs{y}} \le 0$.
According to Lemma~\ref{lemCCinequalities}, the point $(\mbfs{g}^t, \mbfs{b}^t)$ is a conic combination of points in $\{(\Gamma(\mbfs{\beta}),\mbfs{\beta})\ : \mbfs{\beta} \in D^m\}$, say $(\mbfs{g}^t, \mbfs{b}^t) = \sum_{i=1}^k \lambda_i (\Gamma(\mbfs{\beta}^i),\mbfs{\beta}^i)$, where $\lambda_1, \dotsc, \lambda_k > 0$.
Hence, using~\eqref{expI}, we have $\mbfs{g}^t{}^\top \mbfs{x} - \mbfs{b}^t{}^\top \mbfs{y} \ge 0$ with equality if and only if $(\Gamma(\mbfs{\beta}^i),\mbfs{\beta}^i) = (\Gamma(\overline{\mbfs{\beta}}),\overline{\mbfs{\beta}})$ for each $i \in \{1, \dotsc, k\}$.
This shows that $(\mbfs{g}^t, \mbfs{b}^t)  = (\Gamma(\overline{\mbfs{\beta}}),\overline{\mbfs{\beta}})$.
In particular, $\lim_{t\to\infty}(\mbfs{g}^t, \mbfs{b}^t)  = (\Gamma(\overline{\mbfs{\beta}}),\overline{\mbfs{\beta}})$, so $((\overline{\mbfs{x}}, \overline{\mbfs{y}}) )_{t=1}^{\infty}$ is an exposing sequence for $\Gamma(\overline{\mbfs{\beta}})^\top\mbfs{x} - \overline{\mbfs{\beta}}{}^\top\mbfs{y} \ge 0$.

\noindent {\bf Case 2.} Assume that $(\overline{\mbfs{x}}, \overline{\mbfs{y}}) \not \in Q$.
That is, we may assume that 
\begin{equation}\label{eqNotExpansive}
\Delta := \|\overline{\mbfs{x}}\|^2 - \|\overline{\mbfs{y}}\|^2 > 0.
\end{equation}
The proposed exposing sequence for $\Gamma(\overline{\mbfs{\beta}})^\top{\mbfs{x}} - \overline{\mbfs{\beta}}{}{}^\top{\mbfs{y}} \ge 0$ is $((\mbfs{x}^t, \mbfs{y}^t))_{t=1}^{\infty}$, where
\[
\begin{array}{rcl}
\mbfs{x}^t & := & \displaystyle\Gamma(\overline{\mbfs{\beta}}) + \frac{\sqrt{2t+1}}{t} \cdot \overline{\mbfs{x}}\\[.25 cm]
\mbfs{y}^t & := &\displaystyle \left(1+ \frac{4}{t}\Delta \right) \overline{\mbfs{\beta}} + \frac{\sqrt{2t+1}}{t} \cdot \overline{\mbfs{y}}.
\end{array}
\]
We will start the sequence at large enough $t$. 
More precisely, let $\overline{t}$ be large enough so that
\[
6t-8\sqrt{2t+1}\cdot\|\overline{\mbfs{y}}\|-1 \ge 0 \quad \forall ~t \ge \overline{t}.
\]
Our choice to restrict the sequence to $t \ge \overline{t}$ is apparent from the next claim, which shows that the candidate sequence $((\mbfs{x}^t, \mbfs{y}^t))_{t=\overline{t}}^{\infty}$ is indeed a sequence of points in $Q$.

\begin{Claim}
For each $t \ge \overline{t}$, we have $(\mbfs{x}^t, \mbfs{y}^t) \in Q$.
\end{Claim}

\begin{cpf}
Note that
\[
\|\mbfs{x}^t\|^2 = 1+ \frac{2\sqrt{2t+1}}{t} \cdot \Gamma(\overline{\mbfs{\beta}}){}{}^\top\overline{\mbfs{x}}+\frac{2t+1}{t^2} \cdot \overline{\mbfs{x}}{}^\top \overline{\mbfs{x}}.
\]
Using the equation $\Gamma(\overline{\mbfs{\beta}})^\top\overline{\mbfs{x}} - \overline{\mbfs{\beta}}{}{}^\top\overline{\mbfs{y}} = 0$ in~\eqref{expI}, we see that
\[
\|\mbfs{x}^t\|^2 = 1+ \frac{2\sqrt{2t+1}}{t} \cdot \overline{\mbfs{\beta}}{}{}^\top\overline{\mbfs{y}}+\frac{2t+1}{t^2} \cdot \overline{\mbfs{x}}{}^\top \overline{\mbfs{x}}.
\]
Furthermore, we have
\[
\|\mbfs{y}^t\|^2 = \left(1+ \frac{4}{t}\Delta \right)^2 + 2 \left(1+ \frac{4}{t}\Delta \right) \frac{\sqrt{2t+1}}{t} \cdot \overline{\mbfs{\beta}}{}{}^\top\overline{\mbfs{y}} + \frac{2t+1}{t^2} \cdot \overline{\mbfs{y}}{}^\top \overline{\mbfs{y}}.
\]

We would like to show $(\mbfs{x}^t, \mbfs{y}^t) \in Q$ when $t \ge \overline{t}$;
this is equivalent to $\|\mbfs{y}^t\|^2 - \|\mbfs{x}^t\|^2 \ge 0$.
Using the previous two displayed equations, we have
\[
\|\mbfs{y}^t\|^2 - \|\mbfs{x}^t\|^2 = \frac{8}{t}\Delta +\frac{16}{t^2}\Delta^2 + \Delta \frac{8\sqrt{2t+1}}{t^2} \overline{\mbfs{\beta}}{}{}^\top\overline{\mbfs{y}} - \frac{2t+1}{t^2} \Delta.
\]
To show $\|\mbfs{y}^t\|^2 - \|\mbfs{x}^t\|^2 \ge 0$, we can multiply both sides in the previous equation by $\frac{t^2}{\Delta} > 0$ and prove the resulting right-hand side is nonnegative.
That is, it suffices to prove that the quantity
\[
8t+16\Delta
+ 8\sqrt{2t+1} \cdot \overline{\mbfs{\beta}}{}{}^\top\overline{\mbfs{y}}
- (2t+1)
\]
is nonnegative.
By Cauchy-Schwarz, we have $|\overline{\mbfs{\beta}}{}{}^\top\overline{\mbfs{y}}| \le \|\overline{\mbfs{y}}\|\cdot\|\overline{\mbfs{\beta}}\| = \|\overline{\mbfs{y}}\|$.
Hence,
\[
8t+\underbrace{16\Delta}_{\ge 0~\text{by}~\eqref{eqNotExpansive}}
+ 8\sqrt{2t+1} \cdot \underbrace{\overline{\mbfs{\beta}}{}{}^\top\overline{\mbfs{y}}}_{\ge - \|\overline{\mbfs{y}}\|}
- (2t+1) \ge 6t-8\sqrt{2t+1}\cdot\|\overline{\mbfs{y}}\|-1.
\]
The assumption $t \ge \overline{t}$ means that the right-hand side of the previous inequality is nonnegative.
\end{cpf}

Now, consider a sequence of inequalities that are valid for $C_{\Gamma}$ and that separate $((\mbfs{x}^t, \mbfs{y}^t))_{t=\overline{t}}^{\infty}$.
Using Lemma~\ref{lemCCinequalities}, we may assume that the inequality that separates $(\mbfs{x}^t, \mbfs{y}^t)$ has the form
\begin{equation}\label{eqForm1}
\sum_{i=1}^{n+m} \lambda_{[i,t]}\left(\Gamma(\mbfs{\beta}^{[i,t]})^\top\mbfs{x}^t - \mbfs{\beta}^{[i,t]}{}{}^\top \mbfs{y}^t\right) \le 0,
\end{equation}
where each $\lambda_{[i,t]} \ge 0$ and $\mbfs{\beta}^{[i,t]} \in D^m$\footnote{Note that we cannot assume $(\Gamma(\mbfs{\beta}^{[i,t]}),\mbfs{\beta}^{[i,t]})\in I$ because Lemma~\ref{lemCCinequalities} does not hold if $\left\{(\Gamma(\mbfs{\beta}),-\mbfs{\beta})\,: \mbfs{\beta} \in D^m\,\right\}$ is replaced by $I$, which is not closed.  
We will not require that these points be in $I$, though.
Here, we are implicitly using two representations of $C_{\Gamma}$, one obtained by intersecting inequalities whose coefficients belong to $I$ and another obtained by intersecting inequalities whose coefficients belong to $\left\{(\Gamma(\mbfs{\beta}),-\mbfs{\beta})\,: \mbfs{\beta} \in D^m\,\right\}$. }.
Furthermore, in order to demonstrate that $((\mbfs{x}^t, \mbfs{y}^t))_{t=\overline{t}}^\infty$ is an exposing sequence, we assume that all of these separating inequalities have the same norm as $(\Gamma(\overline{\mbfs{\beta}}), \overline{\mbfs{\beta}})$, that is
\begin{equation}\label{eqSameNorm}
\left\|\sum_{i=1}^{n+m} \lambda_{[i,t]}\left(\Gamma(\mbfs{\beta}^{[i,t]}), \mbfs{\beta}^{[i,t]}\right)\right\| = \left\|(\Gamma(\overline{\mbfs{\beta}}), \overline{\mbfs{\beta}})\right\| = \sqrt{2} \qquad \forall ~ t.
\end{equation}

\begin{Claim}\label{claimBoundTau}
    There exists $\tau \in \mbb{R}_+$ such that $\lambda_{[i,t]} \le \tau$ for each $i \in \{1, \dotsc, n+m\}$ and $t \ge \overline{t}$.
\end{Claim}
\begin{cpf}
Assume to the contrary that no $\tau$ exists. 
Then, there exists some $i$, say $i = 1$, and a subsequence of $(\lambda_{[1,t]})_{t=\overline{t}}^{\infty}$ that diverges to $\infty$.
By re-indexing the sequence, we assume that $(\lambda_{[1,t]})_{t=\overline{t}}^{\infty}$ is increasing and diverges to $\infty$.

According to~\eqref{eqSameNorm}, the sequence $(\sum_{i=1}^{n+m} \lambda_{[i,t]}(\Gamma(\mbfs{\beta}^{[i,t]}), \mbfs{\beta}^{[i,t]}))_{t=\overline{t}}^{\infty}$ has a convergent subsequence. 
After re-indexing along this subsequence, we may assume that the limit exists, say 
\[
\lim_{t\to\infty} \sum_{i=1}^{n+m} \lambda_{[i,t]}\left(\Gamma(\mbfs{\beta}^{[i,t]}), \mbfs{\beta}^{[i,t]}\right) =: (\mbfs{g}, \mbfs{b}),
\]
where $\|(\mbfs{g}, \mbfs{b}) \| = \sqrt{2}$.
Furthermore, after taking additional subsequences, we may assume that the limit of $((\Gamma(\mbfs{\beta}^{[1,t]}), \mbfs{\beta}^{[1,t]}))_{t=\overline{t}}^\infty$ exists, say $(\Gamma(\mbfs{\beta}^1), \mbfs{\beta}^1) =: \lim_{t\to\infty} (\Gamma(\mbfs{\beta}^{[1,t]}, \mbfs{\beta}^{[1,t]})$.

We have
\begin{align*}
\mbfs{0} = \lim_{t\to\infty} \frac{1}{\lambda_{[1,t]}} (\mbfs{g}, \mbfs{b}) &= \lim_{t\to\infty} \sum_{i=1}^{n+m} \frac{\lambda_{[i,t]}}{\lambda_{[1,t]}}\left(\Gamma(\mbfs{\beta}^{[i,t]}), \mbfs{\beta}^{[i,t]}\right)\\& = (\Gamma(\mbfs{\beta}^1), \mbfs{\beta}^1)+\lim_{t\to\infty} \sum_{i=2}^{n+m} \frac{\lambda_{[i,t]}}{\lambda_{[1,t]}}\left(\Gamma(\mbfs{\beta}^{[i,t]}), \mbfs{\beta}^{[i,t]}\right).
\end{align*}
Thus, $-(\Gamma(\mbfs{\beta}^1), \mbfs{\beta}^1)=\lim_{t\to\infty} \sum_{i=2}^{n+m} \frac{\lambda_{[i,t]}}{\lambda_{[1,t]}}\left(\Gamma(\mbfs{\beta}^{[i,t]}), \mbfs{\beta}^{[i,t]}\right)$.
Hence, the inequality $-(\Gamma(\mbfs{\beta}^{[1,t]}){}{}^\top \mbfs{x} - \mbfs{\beta}^{[1,t]}{}{}^\top\mbfs{y}) \le 0$ is a limit of valid inequalities for $C_{\Gamma}$, and is therefore itself valid for $C_{\Gamma}$.
However, $\Gamma(\mbfs{\beta}^{[1,t]}){}{}^\top \mbfs{x} - \mbfs{\beta}^{[1,t]}{}{}^\top\mbfs{y} \le 0$ is also valid for $C_{\Gamma}$, implying that $C_{\Gamma}$ is not full-dimensional. 
This is a contradiction.
\end{cpf}

To show that $((\mbfs{x}^t, \mbfs{y}^t))_{t=\overline{t}}^\infty$ is an exposing sequence, we need to show that
\[
\lim_{t \to \infty} \sum_{i=1}^{n+m} \lambda_{[i,t]}\left(\Gamma(\mbfs{\beta}^{[i,t]}), \mbfs{\beta}^{[i,t]}\right) = \left(\Gamma(\overline{\mbfs{\beta}}), \overline{\mbfs{\beta}}\right).
\]
Since the sequence is bounded (see \eqref{eqSameNorm}), it is enough to show that every convergent subsequence converges to $\left(\Gamma(\overline{\mbfs{\beta}}), \overline{\mbfs{\beta}}\right)$.
Abusing notation, let $\left(\sum_{i=1}^{n+m} \lambda_{[i,t]}\left(\Gamma(\mbfs{\beta}^{[i,t]}), \mbfs{\beta}^{[i,t]}\right)\right)_{t=\overline{t}}^\infty$ be our convergent subsequence.
After taking more subsequences, we may assume that for each $i \in \{1, \dotsc, n+m\}$ we have
\[
\lim_{t \to\infty} \left(\Gamma(\mbfs{\beta}^{[i,t]}), \mbfs{\beta}^{[i,t]}\right) =: \left(\Gamma(\mbfs{\beta}^{i}), \mbfs{\beta}^{i}\right),
\]
and furthermore because $\lambda_{[i,t]}$ is uniformly bounded from Claim~\ref{claimBoundTau}, we can also assume
\[
\lim_{t \to\infty} \lambda_{[i,t]} =: \lambda_i.
\]

To prove the limit we are going to exploit the fact that the convergent subsequence satisfies~\eqref{eqForm1}.
Consider multiplying~\eqref{eqForm1} by $t$.
Notice that for each $\mbfs{\beta} \in D^m$ and $t \ge \overline{t}$, we have
\begin{align*}
&t \left(\Gamma(\mbfs{\beta}){}{}^\top \mbfs{x}^t - \mbfs{\beta}{}^\top \mbfs{y}^t\right)\\ =& 
t \left(\Gamma(\mbfs{\beta}){}{}^\top\Gamma(\overline{\mbfs{\beta}}) - \mbfs{\beta}{}^\top\overline{\mbfs{\beta}}\right)
+ \sqrt{2t+1} \left(\Gamma(\mbfs{\beta}){}{}^\top\overline{\mbfs{x}} - \mbfs{\beta}{}^\top\overline{\mbfs{y}}\right)
- 4\Delta\mbfs{\beta}{}^\top \overline{\mbfs{\beta}}.
\end{align*}

Then, multiplying~\eqref{eqForm1} by $t$ can be written as
\begin{equation}\label{eqForm2}
\begin{array}{rccl}
0 &\ge&  t & \displaystyle \left[ \sum_{i=1}^{n+m} \lambda_{[i,t]}\left((\Gamma(\mbfs{\beta}^{[i,t]}){}{}^\top\Gamma(\overline{\mbfs{\beta}}) - \mbfs{\beta}^{[i,t]}{}^\top \overline{\mbfs{\beta}}
\right)\right]\\[.75 cm]
& + &   \sqrt{2t+1} & \displaystyle\left[ \sum_{i=1}^{n+m} \lambda_{[i,t]}\left((\Gamma(\mbfs{\beta}^{[i,t]}){}{}^\top\overline{\mbfs{x}} - \mbfs{\beta}^{[i,t]}{}^\top \overline{\mbfs{y}}
\right)\right]\\[.75 cm]
& - &  & \displaystyle\left[ \sum_{i=1}^{n+m} \lambda_{[i,t]}\left(4\Delta\mbfs{\beta}^{[i,t]}{}^\top \overline{\mbfs{\beta}}\right)\right].
\end{array}
\end{equation}

Recall $\lambda_{[i,t]}$ is uniformly bounded by $\tau$ from Claim~\ref{claimBoundTau}.
Furthermore, because $\Gamma$ is non-expansive, $(\Gamma(\mbfs{\beta}^{[i,t]}){}{}^\top\Gamma(\overline{\mbfs{\beta}}) - \mbfs{\beta}^{[i,t]}{}^\top \overline{\mbfs{\beta}} \geq 0$.
Hence, we can lower bound~\eqref{eqForm2} by
\begin{equation}\label{eqForm3}
0 \ge \sqrt{2t+1}  \left[ \sum_{i=1}^{n+m} \lambda_{[i,t]}\left((\Gamma(\mbfs{\beta}^{[i,t]}){}{}^\top\overline{\mbfs{x}} - \mbfs{\beta}^{[i,t]}{}^\top \overline{\mbfs{y}} \right)\right]
 - \tau \cdot 4\Delta.
\end{equation}

Now, suppose there exists some $i \in \{1, \dotsc, n+m\}$ for which $\mbfs{\beta}^i$ is not equal to $\overline{\mbfs{\beta}}$ and $\lambda_i > 0$.
According to~\eqref{expI}, we then have $\Gamma(\mbfs{\beta}^i){}{}^\top\overline{\mbfs{x}} - \mbfs{\beta}^i{}{}^\top \overline{\mbfs{y}} > 0$.
Under this assumption, the right-hand side of~\eqref{eqForm3} is lower bounded by $\sqrt{2t+1} \cdot \lambda_{[i,t]} \cdot (\Gamma(\mbfs{\beta}^{[i,t]}){}{}^\top\overline{\mbfs{x}} - \mbfs{\beta}^{[i,t]}{}{}^\top \overline{\mbfs{y}})- \tau \cdot 4\Delta$, which tends to $\infty$ as $t \to\infty$.
However, this contradicts the inequality in~\eqref{eqForm3}.
So, if $\mbfs{\beta}^i$ is not equal to $\overline{\mbfs{\beta}}$, then $\lambda_i = 0$.

In conclusion, if $\lambda_i > 0$, then $\mbfs{\beta}^i = \overline{\mbfs{\beta}}$.
Hence, 
\[
\lim_{t \to \infty} \sum_{i=1}^{n+m} \lambda_{[i,t]}\left(\Gamma(\mbfs{\beta}^{[i,t]}), \mbfs{\beta}^{[i,t]}\right) = \left(\Gamma(\overline{\mbfs{\beta}}), \overline{\mbfs{\beta}}\right),
\]
which demonstrates that $((\mbfs{x}^t, \mbfs{y}^t))_{t=\overline{t}}^\infty$ is an exposing sequence for $\Gamma(\overline{\mbfs{\beta}})^\top{\mbfs{x}} - \overline{\mbfs{\beta}}{}{}^\top{\mbfs{y}} \ge 0$.
\end{proof}

%%%%%%%%%%%%%%%%%%%%%%%%%%%%%%%%%%
%%%%%%%%%%%%%%%%%%%%%%%%%%%%%%%%%%
%%%%%%%%%%%%%%%%%%%%%%%%%%%%%%%%%%

\subsection{A necessary condition for maximality}

In this section, we prove the following result.

\begin{lemma}\label{lemSMax}
Let $\Gamma: D^m \to D^n$ and define $C_\Gamma$ as in \eqref{Cgammadef}. 
If $C_{\Gamma}$ is a full-dimensional maximal $Q$-free set, then $\Gamma$ is non-expansive and $(\mbfs{0}, \mbfs{0}) \not\in \conv(\{(\Gamma(\mbfs{\beta}),-\mbfs{\beta})\ : \mbfs{\beta} \in D^m\})$.
\end{lemma}

\proof
Let us first prove that $\Gamma$ is non-expansive. 
By contradiction, assume $\mbfs{\beta}^1,\mbfs{\beta}^2 \in D^m$ are such that
\[\|\Gamma(\mbfs{\beta}^1) - \Gamma(\mbfs{\beta}^2)\| > \|\mbfs{\beta}^1 - \mbfs{\beta}^2\| \Leftrightarrow \mbfs{\beta}^1{}^\top \mbfs{\beta}^2 > \Gamma(\mbfs{\beta}^1)^\top \Gamma(\mbfs{\beta}^2)  \]
This implies that $(\Gamma(\mbfs{\beta}^1), \mbfs{\beta}^1)\not\in C_\Gamma$. 
Define the convex cone
\[
C' := C_\Gamma + \cone((\Gamma(\mbfs{\beta}^1), \mbfs{\beta}^1)),
\]
 where the addition is the Minkowski sum.
 We claim $C'$ is $Q$-free and strictly contains $C_\Gamma$.

The fact that $C_\Gamma \subsetneq C'$ follows because $(\Gamma(\mbfs{\beta}^1), \mbfs{\beta}^1) \in C'\setminus C_\Gamma$.
To prove that $C'$ is $Q$-free, we show that if $(\widetilde{\mbfs{x}},\widetilde{\mbfs{y}})\in \intr(C')$, then $(\widetilde{\mbfs{x}},\widetilde{\mbfs{y}})\not \in Q$.
Given $(\widetilde{\mbfs{x}},\widetilde{\mbfs{y}})\in \intr(C')$, we can write
\[(\widetilde{\mbfs{x}},\widetilde{\mbfs{y}}) = (\mbfs{x},\mbfs{y}) + \lambda (\Gamma(\mbfs{\beta}^1), \mbfs{\beta}^1)\]
for some $(\mbfs{x},\mbfs{y}) \in \intr(C_\Gamma)$ (as $C_\Gamma$ is full-dimensional) and $\lambda > 0$ (see~\cite[Corollary 6.6.2]{R1970}).
From this,
\begin{align*}
\|\widetilde{\mbfs{x}}\|^2  = \|\mbfs{x}\|^2 + \lambda^2 + 2\lambda \Gamma(\mbfs{\beta}^1)^\top \mbfs{x} 
\geq & \|\mbfs{x}\|^2 + \lambda^2 + 2\lambda \mbfs{\beta}^1{}^\top \mbfs{y} \\
 >& \|\mbfs{y}\|^2 + \lambda^2 + 2\lambda \mbfs{\beta}^1{}^\top \mbfs{y} \\
 = &\|\widetilde{\mbfs{y}}\|^2,
\end{align*}
where the first inequality holds because $(\mbfs{x},\mbfs{y}) \in C_\Gamma $ and the second inequality holds because $C_\Gamma$ is $Q$-free.
Therefore, $(\widetilde{\mbfs{x}},\widetilde{\mbfs{y}})\not \in Q$ and $C'$ is $Q$-free. 
However, this contradicts the maximality of $C_\Gamma$.

Now that we have established that $\Gamma$ is non-expansive, we see that it is continuous. 
Thus, we can use Lemma \ref{lemNoZeroSameCoeff} to conclude that $(\mbfs{0}, \mbfs{0}) \not\in \conv(\{(\Gamma(\mbfs{\beta}),-\mbfs{\beta})\ : \mbfs{\beta} \in D^m\})$.
\endproof

%%%%%%%%%%%%%%%%%%%%%%%%
%%%%%%%%%%%%%%%%%%%%%%%%

\section{A proof of Theorem~\ref{thmCpolyhedral}}\label{secPolyhedral}

$(\Leftarrow)$ We show that if $\overline{\mbfs{\beta}}\in D^m \setminus I$, then $\Gamma(\overline{\mbfs{\beta}}){}^\T \mbfs{x} - \overline{\mbfs{\beta}}{}^\T \mbfs{y} \ge 0$ is implied by the inequalities indexed by $I$. 
Let $\overline{\mbfs{\beta}}\in D^m \setminus I$.
By assumption, there exists a set $J \subseteq I$ of pairwise isometric points satisfying $\overline{\mbfs{\beta}} \in \cone(J)$.
Hence, there exist $\lambda_{\mbfs{\beta}} \ge 0$ for each $\mbfs{\beta} \in J$ such that $\overline{\mbfs{\beta}} = \sum_{\mbfs{\beta} \in J} \lambda_{\mbfs{\beta}} \mbfs{\beta}$. 
We have $\Gamma(\overline{\mbfs{\beta}}) =  \sum_{\mbfs{\beta} \in J} \lambda_{\mbfs{\beta}}\Gamma(\mbfs{\beta})$ by Lemma~\ref{lemIsoSets}.
This shows that $(\Gamma(\overline{\mbfs{\beta}}), -\overline{\mbfs{\beta}}) \in \cone(\{(\Gamma(\mbfs{\beta}), -\mbfs{\beta}): \mbfs{\beta}\in J\})$.
Hence, $\Gamma(\overline{\mbfs{\beta}})^\T \mbfs{x} - \overline{\mbfs{\beta}}{}^\T \mbfs{y} \ge 0$ is implied by the inequalities indexed by $I$.

\smallskip
\noindent $(\Rightarrow)$ $C_\Gamma$ is a polyhedron, so by Lemma~\ref{lemCCinequalities} there is a finite representation
\[
C_\Gamma = \left\{(\mbfs{x},\mbfs{y}) \in \mbb{R}^n \times \mbb{R}^n:\ \Gamma(\mbfs{\beta})^\T \mbfs{x} - {\mbfs{\beta}}{}^\T \mbfs{y} \ge 0 ~~ \forall\ \mbfs{\beta} \in I\right\}.
\]
Let $\overline{\mbfs{\beta}}\in D^m \setminus I$.
The inequality $\Gamma(\overline{\mbfs{\beta}})^\T\mbfs{x} -\overline{\mbfs{\beta}}{}^\T \mbfs{y} \ge 0$ is valid for $C_\Gamma$, so there exists a set $J \subseteq I$ and positive coefficients $\lambda_{\mbfs{\beta}}$ for each $\mbfs{\beta} \in J$ such that
\[
(\Gamma(\overline{\mbfs{\beta}}),\overline{\mbfs{\beta}}) = \sum_{\mbfs{\beta}\in J}\lambda_{\mbfs{\beta}} (\Gamma(\mbfs{\beta}),\mbfs{\beta}).
\]
Lemma~\ref{lempolyhedraIsometricFacets} states that
\(
\overline{\mbfs{\beta}}{}^\T \mbfs{\beta} = \Gamma(\overline{\mbfs{\beta}})^\T\Gamma(\mbfs{\beta})
\)
for all $\mbfs{\beta} \in J$. 
For each $\mbfs{\beta}'\in J$, we have
\[
 \Gamma(\mbfs{\beta}')^\T\Gamma(\overline{\mbfs{\beta}})  - \mbfs{\beta}'{}^\T \overline{\mbfs{\beta}}= \sum_{\mbfs{\beta}\in J}\lambda_{\mbfs{\beta}} (\Gamma(\mbfs{\beta}')^\T \Gamma({\mbfs{\beta}})-\mbfs{\beta}'{}^\T {\mbfs{\beta}}  ).
\]
The left-hand side is 0 because $\overline{\mbfs{\beta}}$ and $\mbfs{\beta}'$ are isometric, and every summand on the right-hand side is nonnegative because $\lambda_{\mbfs{\beta}} >0$ and $\Gamma(\mbfs{\beta}')^\T \Gamma({\mbfs{\beta}}) -\mbfs{\beta}'{}^\T {\mbfs{\beta}}  \ge 0$ by the non-expansive property of $\Gamma$.
Hence, $\Gamma(\mbfs{\beta}')^\T \Gamma(\mbfs{\beta}) = \mbfs{\beta}'{}^\T {\mbfs{\beta}} $ for all $\mbfs{\beta} \in J$.
As ${\mbfs{\beta}}'$ was arbitrarily chosen in $J$, we see that all elements of $J$ are pairwise isometric and $\overline{\mbfs{\beta}} \in \cone(J)$.
%

%%%%%%%%%%%%%%%%%%
%%%%%%%%%%%%%%%%%%

\section{Future Work}\label{secFutureWork}

Even though we have fully characterized full-dimensional maximal $Q$-free sets, there are multiple lines of future work on maximal quadratic-free sets.

One future line of work we consider interesting is related to maximal polyhedral $Q$-free sets.
Theorem~\ref{thmCpolyhedral} characterizes when $C_{\Gamma}$ is polyhedral;
loosely speaking, it reduces polyhedrality to finding a set of isometric points that cover $D^m$.
With this in mind, one may consider whether an arbitrary finite set of points in $D^m$ can be extended to an `isometric cover' for some function $\Gamma$. 
The following example illustrates that this is not always possible.
\begin{example}
Set $I := \{\mbfs{\beta}^1, \mbfs{\beta}^2, \mbfs{\beta}^3, \mbfs{\beta}^4\} \subseteq D^2$, where 
\[
\mbfs{\beta}^1 := \big(0,1\big),~ \mbfs{\beta}^2 := \big(1,0\big),~ \mbfs{\beta}^3 := \big(0,-1\big),~\text{and}~\mbfs{\beta}^4 := \big(-\sqrt{\sfrac{2}{3}}, - \sqrt{\sfrac{1}{3}}\big).
\]
The points are drawn in the next figure. 

\begin{center}
\begin{tikzpicture}[scale = .75, every node/.style={scale=0.75}, baseline = 3]

\draw[black!50](-1.15,0)--(1.15,0);
\draw[black!50](0,-1.15)--(0,1.15);
\draw[draw=black!50, fill=none](0,0) circle (1);

\draw[MediumRed, ultra thick,domain=0:90] plot ({cos(\x)}, {sin(\x)});
\draw[MediumPurple, ultra thick,domain=270:360] plot ({cos(\x)}, {sin(\x)});
\draw[MediumBlue, ultra thick,domain=210:270] plot ({cos(\x)}, {sin(\x)});
\draw[JungleGreen, ultra thick,domain=90:210] plot ({cos(\x)}, {sin(\x)});

\draw (90:.85) -- (90:1.15) node[above = .1]{$\mbfs{\beta}^1$};
\draw (0:.85) -- (0:1.15) node[right = .1]{$\mbfs{\beta}^2$};
\draw (-90:.85) -- (-90:1.15) node[below = .1]{$\mbfs{\beta}^3$};
\draw (210:.85) -- (210:1.15) node[below left = .1]{$\mbfs{\beta}^4$};

\end{tikzpicture}
\end{center}
We claim that the set $I$ cannot be used in Theorem~\ref{thmCpolyhedral} to define a non-expansive map $\Gamma$ that corresponds to a full-dimensional $Q$-free polyhedron $C_{\Gamma}$.
Indeed, if $I$ could be used to construct such a $\Gamma$, then the $I$ must cover $D^2$ with isometric pairs of points, that is, each consecutive pair of $\mbfs{\beta}^i$ and $\mbfs{\beta}^{i+1}$ in $I$ must be isometric. 
Thus, the image of the arc $(\mbfs{\beta}^1, \mbfs{\beta}^2)$ under $\Gamma$ must be a rotation of the red arc; similarly the image of the purple arc under $\Gamma$ must be a rotation of the purple arc. 
Given that $C_{\Gamma}$ is full-dimensional, we then conclude that $\Gamma(\mbfs{\beta}^1) = \Gamma(\mbfs{\beta}^3)$ otherwise we would have $(\mbfs{0}, \mbfs{0})\in \conv(\{(\Gamma(\mbfs{\beta}),-\mbfs{\beta})\,:\ \mbfs{\beta}\in D^2\})$ and contradict Theorem~\ref{thm:continuous}.
However, the blue (respectively, the green) arc must be mapped by $\Gamma$ to an arc of the same length. 
The equation $\Gamma(\mbfs{\beta}^1) = \Gamma(\mbfs{\beta}^3)$ then shows that the blue and green arcs must have the same length, which is not true.
\end{example}

We believe that finding a characterization of when a set of points in $D^m$ can be used to define a maximal $Q$-free polyhedron is an important follow-up question in the understanding of maximal $Q$-free sets.

Beyond the homogeneous setting, it is interesting and potentially impactful to consider the non-homogeneous case. 
For instance, which maximal sets in the homogeneous setting are maximal for the non-homogeneous setting, or how can one lift maximal $Q$-free sets to handle the non-homogeneous setting. 

\section*{Statements and Declaration.}
\noindent{\bf Funding.} J. Paat was supported by a Natural Sciences and Engineering Research Council of Canada Discovery Grant [RGPIN-2021-02475]. G. Mu\~noz was supported by the National Research and Development Agency of Chile (ANID) through the Fondecyt Grant 1231522.

\medskip

\noindent{\bf Competing interests.}
The authors have no relevant financial or non-financial interests to disclose.

\medskip

\noindent{\bf Compliance with Ethical Standards.}
The authors have no conflicts of interest to declare that are relevant to the content of this article.

\medskip

\noindent{\bf Acknowledgements}
The authors would like to thank Roberto Cominetti for fruitful discussions; in particular, for discussions leading to Remark \ref{rmk:contracting}, for noticing the subtlety in Theorem 17.3 from~\cite{R1970} discussed in Remark \ref{rmk:rockafellar}, and for the example provided in the latter.

%\newpage
\bibliographystyle{spmpsci}
\bibliography{references}

\begin{thebibliography}{10}
\providecommand{\url}[1]{{#1}}
\providecommand{\urlprefix}{URL }
\expandafter\ifx\csname urlstyle\endcsname\relax
  \providecommand{\doi}[1]{DOI~\discretionary{}{}{}#1}\else
  \providecommand{\doi}{DOI~\discretionary{}{}{}\begingroup
  \urlstyle{rm}\Url}\fi

\bibitem{AJ2013}
Andersen, K., Jensen, A.: Intersection cuts for mixed integer conic quadratic
  sets.
\newblock In: Integer Programming and Combinatorial Optimization, pp. 37--48.
  Springer (2013)

\bibitem{ALWW2007}
Andersen, K., Louveaux, Q., Weismantel, R., Wolsey, L.: Cutting planes from two
  rows of the simplex tableau.
\newblock Proceedings of Integer Programming and Combinatorial Optimization
  (IPCO) pp. 1--15 (2007)

\bibitem{averkov2013maximal}
Averkov, G.: On maximal s-free sets and the helly number for the family of
  s-convex sets.
\newblock SIAM Journal on Discrete Mathematics \textbf{27}(3), 1610--1624
  (2013)

\bibitem{A2013}
Averkov, G.: A proof of {L}ov\'{a}sz's theorem on maximal lattice-free sets.
\newblock Contributions to Algebra and Geometry  (2013)

\bibitem{ABP2018}
Averkov, G., Basu, A., Paat, J.: Approximation of corner polyhedra with
  families of intersection cuts.
\newblock {SIAM} Journal on Optimization \textbf{28}(1), 904--929 (2018)

\bibitem{BOW2016}
Baes, M., Oertel, T., Weismantel, R.: Duality for mixed-integer convex
  minimization.
\newblock {Mathematical Programming} \textbf{158}, 547--564 (2016)

\bibitem{B1971}
Balas, E.: Intersection cuts - a new type of cutting planes for integer
  programming.
\newblock Operations Research  (1971)

\bibitem{BCCWW2017}
Basu, A., Conforti, M., Cornu{\'e}jols, G., Weismantel, R., Weltge, S.:
  Optimality certificates for convex minimization and {H}elly numbers.
\newblock Operations Research Letters \textbf{45}(6), 671--674 (2017)

\bibitem{BCCZ2010}
Basu, A., Conforti, M., Cornu\'{e}jols, G., Zambelli, G.: Maximal lattice-free
  convex sets in linear subspaces.
\newblock Mathematics of Operations Research \textbf{35}(3), 704--720 (2010)

\bibitem{basu2010minimal}
Basu, A., Conforti, M., Cornu{\'e}jols, G., Zambelli, G.: Minimal inequalities
  for an infinite relaxation of integer programs.
\newblock SIAM Journal on Discrete Mathematics \textbf{24}(1), 158--168 (2010)

\bibitem{BDP2019}
Basu, A., Dey, S., Paat, J.: Nonunique lifting of integer variables in minimal
  inequalities.
\newblock {SIAM} Journal on Discrete Mathematics  (2019)

\bibitem{BCM2019}
Bienstock, D., Chen, C., Mu{\~n}oz, G.: Intersection cuts for polynomial
  optimization.
\newblock In: A.~Lodi, V.~Nagarajan (eds.) Integer Programming and
  Combinatorial Optimization, pp. 72--87. Springer International Publishing,
  Cham (2019)

\bibitem{BCM2020}
Bienstock, D., Chen, C., Mu{\~n}oz, G.: Outer-product-free sets for polynomial
  optimization and oracle-based cuts.
\newblock Mathematical Programming \textbf{183}, 105--148 (2020)

\bibitem{chmiela2022implementation}
Chmiela, A., Mu{\~n}oz, G., Serrano, F.: On the implementation and
  strengthening of intersection cuts for qcqps.
\newblock Mathematical Programming pp. 1--38 (2022)

\bibitem{CCDLM2014}
Conforti, M., Cornu\'{e}jols, G., Daniilidis, A., Lemar\'{e}chal, C., Malick,
  J.: Cut-generating functions and {S}-free sets.
\newblock Mathematics of Operations Research  (2014)

\bibitem{CCZ2011}
Conforti, M., Cornu\'{e}jols, G., Zambelli, G.: A geometric perspective on
  lifting.
\newblock Operations Research \textbf{59}(3), 569--577 (2011)

\bibitem{CCZ2014}
Conforti, M., Cornu{\'e}jols, G., Zambelli, G.: Integer Programming.
\newblock Springer (2014)

\bibitem{conforti2016maximal}
Conforti, M., Summa, M.D.: Maximal s-free convex sets and the helly number.
\newblock SIAM Journal on Discrete Mathematics \textbf{30}(4), 2206--2216
  (2016)

\bibitem{DW2010}
Dey, S., Wolsey, L.: Two row mixed-integer cuts via lifting.
\newblock Mathematical Programming \textbf{124}, 143--174 (2010)

\bibitem{L1989}
Lov\'{a}sz, L.: Geometry of numbers and integer programming.
\newblock In: M.Iri, K.~Tanabe (eds.) Mathematical Programming: Recent
  Developments and Applications, pp. 177 -- 201. Kluwer Academic Publishers
  (1989)

\bibitem{MKV2016}
Modaresi, S., K{\i}l{\i}n{\c{c}}, M., Vielma, J.: {Intersection cuts for
  nonlinear integer programming convexification techniques for structured
  sets}.
\newblock {Mathematical Programming}  (2016)

\bibitem{MPS2023}
Mu{\~{n}}oz, G., Paat, J., Serrano, F.: Towards a characterization of maximal
  quadratic-free sets.
\newblock In: A.~Del~Pia, V.~Kaibel (eds.) Integer Programming and
  Combinatorial Optimization, pp. 334--347. Springer International Publishing,
  Cham (2023)

\bibitem{MS2020}
Mu{\~n}oz, G., Serrano, F.: Maximal quadratic-free sets.
\newblock Proceedings of the International Conference on Integer Programming
  and Combinatorial Optimization pp. 307--321 (2020)

\bibitem{MS2022}
Mu{\~n}oz, G., Serrano, F.: Maximal quadratic-free sets.
\newblock Mathematical Programming \textbf{192}, 229--270 (2022)

\bibitem{PSS2022}
Paat, J., Schl{\"o}ter, M., Speakman, E.: Constructing lattice-free gradient
  polyhedra in dimension two.
\newblock Mathematical Programming \textbf{192}(1), 293--317 (2022)

\bibitem{R1970}
Rockafellar, R.T.: Convex analysis.
\newblock Princeton university press (1970)

\bibitem{T1964}
Tuy, H.: Concave minimization under linear constraints with special structure.
\newblock {Doklady Akademii Nauk} \textbf{159}, 32--35 (1964)

\end{thebibliography}

\end{document}